\documentclass[final]{amsart}

\usepackage{booktabs}
\usepackage{fullpage}
\usepackage{tristan}
\usepackage[giveninits=true,eprint=false,url=false]{biblatex}
\usepackage{subcaption}
\graphicspath{{images/}}
\usepackage{algorithm}
\usepackage{algpseudocode}

\title{Characteristic Sweeps and Source Iteration for Charged-Particle Transport with Continuous Slowing-Down and Angular Scattering}

\author[1]{Ben S. Ashby}

\author[3]{Alex Lukyanov}

\author[1,2]{Tristan Pryer}

\address{$^1$ Institute for Mathematical Innovation\\ University of
  Bath, Bath, UK. $^2$ Department of Mathematical Sciences
  \\ University of Bath, Bath, UK.  $^3$ Department of Mathematics and
  Statistics\\ University of Reading, Reading, UK. }

\begin{document}

\begin{abstract}
  We develop a semi-analytic deterministic framework for
  charged-particle transport with continuous slowing-down in energy
  and angular scattering. Directed transport and energy advection are
  treated by method-of-characteristics integration, yielding explicit
  directional sweeps defined by characteristic maps and inflow
  data. Scattering is incorporated through a fixed-point
  (source-iteration) scheme in which the angular gain is lagged,
  yielding a sequence of decoupled directional solves coupled only
  through angular sums.

  The method is formulated variationally in a transport graph space
  adapted to the charged particle drift. Under standard monotonicity
  and positivity assumptions on the stopping power and boundedness
  assumptions on cross sections, we establish coercivity and
  boundedness of the transport bilinear form, prove contraction of the
  source iteration under a subcriticality condition and derive a
  rigorous a posteriori bound for the iteration error, providing an
  efficient stopping criterion.

  We further analyse an elastic discrete-ordinates approximation,
  including conservation properties and a decomposition of angular
  error into quadrature, cone truncation and finite iteration
  effects. Numerical experiments for proton transport validate the
  characteristic sweep against an exact ballistic benchmark and
  demonstrate the predicted fixed-point convergence under
  forward-peaked scattering. Carbon-ion simulations with tabulated
  stopping powers and a reduced multi-species coupling illustrate
  Bragg peak localisation and distal tail formation driven by
  secondary charged fragments.
\end{abstract}

\maketitle

\section{Introduction}

Charged-particle beam therapies (protons, carbon ions) and radiation
physics are naturally described by a Boltzmann transport equation
(BTE) in position, direction and energy, coupling continuous
slowing-down with angular scattering in tissue. In therapy-relevant
regimes this produces two numerical challenges: strong anisotropy in
angle and a depth-localised Bragg peak generated by energy
loss. Efficient deterministic solvers with quantitative stability and
error control are therefore attractive when repeated forward solves
are required, for example in optimisation, robustness assessment or
uncertainty studies \cite{stammer2024deterministic,yu2025review}.  At
the microscopic level, the dominant mechanisms can be grouped as
ionisation losses, elastic angular deflections and rarer nonelastic
nuclear events. Figure~\ref{fig:atom} shows these schematically for a
proton (with analogous mechanisms for heavier ions).

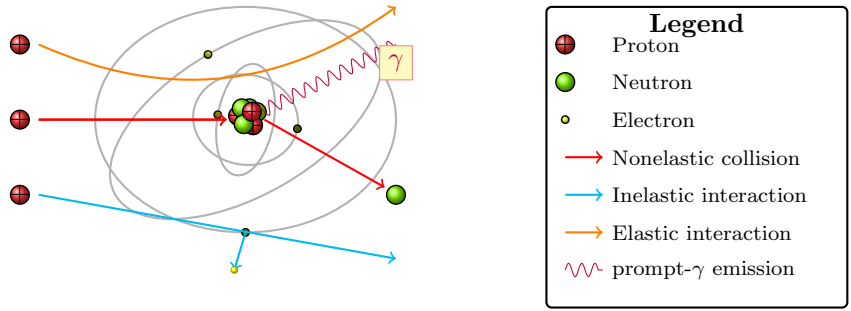
\begin{figure}[h!]
    \centering
    \renewcommand{\proton}[1]{%
    \shade[ball color=red!80!white, draw=black, line width=0.5pt] (#1) circle (.25);
    \draw[black] (#1) node{$+$};
}

\renewcommand{\neutron}[1]{%
    \shade[ball color=lime!70!green, draw=black, line width=0.5pt] (#1) circle (.25);
}

\renewcommand{\electron}[3]{%
    \draw[rotate = #3, color=gray!60!white, line width=0.8pt] (0,0) ellipse (#1 and #2);
    \shade[ball color=yellow!90!white, draw=black, line width=0.5pt] (0,#2)[rotate=#3] circle (.1);
}

\newcommand{\legendelectron}[1]{%
    \shade[ball color=yellow!90!white, draw=black, line width=0.5pt] (#1) circle (.1);
}

\renewcommand{\nucleus}{%
    \neutron{0.1,0.3}
    \proton{0,0}
    \neutron{0.3,0.2}
    \proton{-0.2,0.1}
    \neutron{-0.1,0.3}
    \proton{0.2,-0.15}
    \neutron{-0.05,-0.12}
    \proton{0.17,0.21}
}

\renewcommand{\inelastic}[2]{
  \proton{#1,#2};
  \draw[->,thick,cyan!80!white](#1+0.5,#2)--(4,-3.7); 
  \draw[->,thick,cyan!80!white](0,-3)--(-0.3,-4); 
  \shade[ball color=yellow] (-0.3,-4) circle (.1); 
}

\renewcommand{\elastic}[2]{
  \proton{#1,#2};
  \draw[->,thick,orange,bend right=90](#1+0.5,#2) to  [out=-30, in=-150] (4,3.);
}

\renewcommand{\protoncollision}[3]{
  \proton{#1,#2};
  \draw[->,thick,red](#1+0.5,#2)--(-0.5,0);%
  \draw[snake=coil, line after snake=0pt, segment aspect=0,%
    segment length=5pt,color=red!80!blue] (0,0)-- +(4,2)%
  node[fill=white!70!yellow,draw=red!50!white, below=.01cm,pos=1.]%
  {$\gamma$};%
  \draw[->,thick,red](#1+0.5,#2)--(-0.5,0);%
  \draw[->,thick,red](0.5,0)--(3.7,-1.8);%
  \neutron{4,-2};  
}

\begin{tikzpicture}[scale=0.5]
    \nucleus
    \electron{1.5}{0.75}{80}
    \electron{1.2}{1.4}{260}
    \electron{4}{2}{30}
    \electron{4}{3}{180}
    \protoncollision{-6.}{0.}{160}
    \inelastic{-6.}{-2.}
    \elastic{-6.}{2.}

      \begin{scope}[shift={(8,-1)}]        
        \draw [thick,rounded corners=2pt] (0,-4) rectangle (8,4); 
        \node at (4, 3.5) {\textbf{Legend}};        
        \proton{0.5, 3}
        \node[anchor=west, font=\footnotesize] at (1.5,3) {Proton}; 
        \neutron{0.5, 2}
        \node[anchor=west, font=\footnotesize] at (1.5,2) {Neutron};         
        \legendelectron{0.5, 1}
        \node[anchor=west, font=\footnotesize] at (1.5,1) {Electron};        
        \draw[->,thick,red] (0.5,0) -- +(1,0);
        \node[anchor=west, font=\footnotesize] at (1.5,0) {Nonelastic collision};        
        \draw[->,thick,cyan!80!white] (0.5,-1) -- +(1,0);
        \node[anchor=west, font=\footnotesize] at (1.5,-1) {Inelastic interaction};         
        \draw[->,thick,orange] (0.5,-2) -- +(1,0);
        \node[anchor=west, font=\footnotesize] at (1.5,-2) {Elastic interaction}; 
        \draw[snake=coil, line after snake=0pt, segment aspect=0,
          segment length=5pt,color=red!80!blue] (0.5,-3) -- +(1,0);
        \node[anchor=west, font=\footnotesize] at (1.5,-3) {prompt-$\gamma$ emission};
      \end{scope}
\end{tikzpicture}
    \caption{\em The three main interactions of a proton with matter.
      A \textcolor{red}{nonelastic} proton--nucleus collision, an
      \textcolor{cyan!80!white}{inelastic} Coulomb interaction with
      atomic electrons and \textcolor{orange}{elastic} Coulomb
      scattering with the nucleus.
      \label{fig:atom}
    }
\end{figure}

Monte Carlo simulation remains the reference approach for
high-fidelity charged-particle transport, but its cost can be
prohibitive when many forward solves are required
\cite{kan2013review}. Deterministic methods (discrete-ordinates,
moment and pencil-beam-type schemes) offer the prospect of faster,
reproducible solves, yet they must address the high dimensionality of
phase space and the strong coupling between energy loss and scattering
\cite{stammer2024deterministic,yu2025review}. Recent semi-analytic
deterministic solvers show that characteristic integration of
continuous slowing-down can deliver fast, reproducible proton
transport in regimes where ionisation dominates
\cite{ashby2025efficient}. A key open difficulty is to retain the same
computational advantages while treating angular redistribution
explicitly and with quantitative control of the scattering coupling.

We present a semi-analytic method-of-characteristics (MoC) framework
for charged-particle transport with continuous slowing-down in energy
and angular scattering. Directed transport and energy advection are
handled by characteristic integration, so each directional sweep
amounts to evaluating an explicit characteristic map together with the
prescribed inflow data. Scattering is included through a
source-iteration (fixed-point) strategy in which the angular gain is
lagged, so that each iteration requires only a collection of
direction-by-direction sweeps coupled through angular sums. This
preserves the sweep structure of ionisation-only characteristic
solvers while adding scattering through an iterative angular mixing
step.

We place the method in a variational setting based on a graph space
adapted to the directional and energy drift. Under monotonicity and
positivity assumptions on the stopping power and boundedness
assumptions on the cross sections, we establish coercivity and
boundedness of the transport bilinear form and obtain well-posedness
of the weak problem. These stability results support convergence of
the source iteration and yield rigorous a posteriori bounds for the
iteration error, giving a practical stopping criterion. For the
discrete-ordinates approximation we further decompose the angular
error into contributions from angular quadrature, cone truncation and
the finite iteration count.

We implement the framework for proton transport with
Bragg--Kleeman-type stopping and for carbon-ion transport using
tabulated stopping powers. In the carbon case we include a reduced
multi-species construction in which a secondary proton and neutron
field is generated by a volumetric source driven by the converged
carbon solution and an effective interaction rate, to provide a
transparent mechanism for post-peak dose tails and to validate the
coupling logic. Numerical experiments verify the characteristic solver
against an analytic benchmark, demonstrate linear convergence of the
source iteration and illustrate the impact of forward-peaked
scattering and the multi-species coupling on depth--dose, including
the distal secondary tail in the carbon simulations.

Charged--particle transport spans rigorous kinetic theory,
high--fidelity stochastic simulation and fast deterministic or
surrogate algorithms.  Foundational treatments of the linear Boltzmann
equation and of transport operators provide the analytical setting for
scattering kernels, conservation, boundary conditions and angular
representations
\cite{case1968linear,lewis1984computational,pomraning1991linear,wilson1988charged}.
Reliable stopping--power and range--energy data underpin dose
prediction across modalities; recommended datasets and reviews
document proton and heavy--ion stopping powers, dosimetry constants
and corrections beyond simple power laws
\cite{berger1982stopping,berger1998stopping,international1993stopping,seltzer2016key,sigmund2009errata}.
Against this backdrop, high dimensionality (space, direction, energy),
strongly forward--peaked scattering and continuous slowing--down
combine to make full BTE solutions challenging in therapy--relevant
geometries and materials.

Monte Carlo remains the reference approach for transport with detailed
physics and is widely implemented in general--purpose toolkits and
transport libraries
\cite{agostinelli2003geant4,goorley2012initial,lux2018monte}. Clinical
dose calculation has benefited from proton--specific MC
implementations and accelerations tailored to voxelised media and beam
modelling \cite{fippel2004monte}. Despite these advances,
computational expense can limit throughput in planning and uncertainty
studies and surveys of grid--based and MC algorithms emphasise the
trade--off between fidelity and speed in clinical settings
\cite{kan2013review}.

Deterministic discretisations provide complementary capabilities.
Discrete--ordinates ($S_N$) and spherical--harmonics ($P_N$/$SP_N$)
methods discretise or expand the angular dependence and are
extensively reviewed for nuclear and radiation applications, including
stability considerations for strongly anisotropic scattering
\cite{yu2025review}.  Pencil--beam and Fokker--Planck reductions offer
further simplification for forward--peaked beams and accuracy analyses
quantify regimes where such approximations capture the parent BTE and
where they break down, for example under large--angle scatter or in
heterogeneous media \cite{borgers1996accuracy}. Algorithmically, fast
beam models for therapy incorporate upstream devices and air gaps and
remain influential for treatment planning workflows
\cite{hong1996pencil}.
In therapy settings, forward--peaked scattering is often represented
through simple parametric phase functions such as the Henyey--Greenstein
kernel \cite{henyey1941diffuse}, which provides a convenient testbed for
discrete--ordinates and iterative scattering solvers.

Recent developments aim to improve scalability and structure
preservation for charged--particle transport. Dynamical low--rank
formulations reduce angular--energy complexity by evolving a factored
representation with collided/uncollided splitting, providing
deterministic alternatives to full $S_N$ in high dimensions
\cite{stammer2024deterministic}. Positivity--preserving finite element
frameworks enforce nonnegativity and conservation in deterministic
dose computation, improving robustness in heterogeneous media and
complex geometries \cite{ashby2025positivity}. Efficient analytic or
semi--analytic solvers have been derived for proton transport by
combining characteristic integration of the slowing--down term with
tractable scattering closures \cite{ashby2025efficient}.  High--order
space--angle--energy discontinuous Galerkin formulations on polytopic
meshes provide a route to accurate deterministic Boltzmann solves on
complex geometries \cite{houston2024efficient}. The resulting large
coupled linear systems motivate scalable iterative solvers and
computable a posteriori algebraic error control
\cite{houston2024iterative}.  For sweeping on complex meshes,
cycle--free strategies on polytopal grids enable DG--style
discretisations with guaranteed orderings, extending transport solvers
beyond structured meshes \cite{calloo2025cycle,Evans2026}.

There is parallel progress in modelling, inversion and uncertainty. A
unified operator--theoretic perspective links transport descriptions
to planning objectives and controls, clarifying the interface between
forward models, adjoints and optimisation in proton therapy
\cite{kyprianou2025unified}. Stochastic differential equation
formulations provide mesoscopic descriptions with
structure--preserving discretisations encoding geometry, energy loss,
angular diffusion and sensitivity
\cite{crossley2025jump,dean2025fast,chronholm2025geometry}. Data--driven
surrogates with uncertainty quantification offer rapid dose emulation
with quantified predictive spread \cite{pim2025surrogate}. Bayesian
inverse formulations demonstrate how prompt gamma or related signals
can be used to verify delivered dose fields, integrating transport
models with statistical inference for on--line quality assurance
\cite{cox2024bayesian}.

The semi--analytic approach developed here sits between high--fidelity
MC transport and heavily reduced beam models. It retains an explicit
angular description of scattering while exploiting characteristic
integration of continuous slowing--down in energy and it is designed
to interface naturally with conservation--aware discretisations and
with uncertainty or inverse formulations. This positioning is
consistent with the classical transport literature
\cite{case1968linear,lewis1984computational,pomraning1991linear,wilson1988charged},
the data foundations for stopping power and dosimetry
\cite{berger1982stopping,berger1998stopping,international1993stopping,seltzer2016key,sigmund2009errata},
and contemporary advances in scalable deterministic and hybrid
transport algorithms
\cite{yu2025review,stammer2024deterministic,ashby2025efficient,ashby2025positivity,calloo2025cycle},
with practical relevance for treatment planning and verification.
From a numerical-transport perspective, the present work also draws on
classical analyses of inflow trace spaces for kinetic transport
\cite{cessenat1984theoremes,cessenat1985theoremes} and on the
extensive literature on source iteration and its convergence
properties in discrete ordinates transport
\cite{adams2002fast,miller2002convergence}.

The rest of the paper is set out as follows. In \S\ref{sec:models} we
introduce the charged--particle transport models considered in this
work, define the phase--space setting and specify the scattering and
slowing--down structure that motivates the subsequent analysis. In
\S\ref{sec:variational} we derive a variational formulation in an
appropriate graph space, establish coercivity and boundedness of the
continuous transport bilinear form and prove convergence of the
continuum source iteration. In \S\ref{sec:angular} we develop the
discrete--ordinates angular approximation and discuss conservation
properties of the discrete scattering operator. In
\S\ref{sec:errorcontrol} we give error control that separates angular
quadrature, cone truncation and iteration effects.  In
\S\ref{sec:numericsproton} we present numerical experiments for proton
transport, including an exact benchmark for the ballistic limit,
convergence tests for the source iteration and simulations with
Henyey--Greenstein scattering. In \S\ref{sec:numericscarbon} we extend
the implementation to carbon ions using tabulated stopping powers and
a reduced multi--species construction generating secondary protons and
neutrons, and we illustrate the resulting depth--dose behaviour
including distal tail formation beyond the primary carbon peak.

\section{Charged Particle Transport Models}
\label{sec:models}

We consider transport of charged particles through a medium, governed
by directed motion, ionising energy loss and scattering due to elastic
and nonelastic interactions. Let $D \subset \mathbb{R}^d$ with
$d \in \{2,3\}$ denote the spatial domain, and let
$I = [E_{\min},E_{\max}] \subset (0,\infty)$ be the admissible energy
interval. The phase-space domain is $\Omega = D \times I \times
\mathbb{S}^{d-1}$, where $\vec{\omega} \in \mathbb{S}^{d-1}$ is the
unit direction of travel. The surface measure on $\mathbb{S}^{d-1}$ is
denoted $\d\vec{\omega}$.

For a given particle species, the phase-space density
$\psi(x,E,\vec{\omega})$ satisfies the linear Boltzmann transport
equation
\begin{equation}
  \label{eq:BTE-general}
  \vec{\omega}\cdot\nabla_x \psi(x,E,\vec{\omega})
  + \sigma_T(E)\psi(x,E,\vec{\omega})
  =
  \mathcal{K}[\psi](x,E,\vec{\omega}),
\end{equation}
where $\sigma_T(E)\ge 0$ is the total removal cross-section and
$\mathcal{K}$ is the gain operator
\begin{equation}
  \label{eq:Kop-general}
  \mathcal{K}[\psi](x,E,\vec{\omega})
  :=
  \int_{\mathbb{S}^{d-1}} \int_{I}
  \sigma_S(\vec{\omega},\vec{\omega}',E' \to E)
  \psi(x,E',\vec{\omega}')  \d E'  \d\vec{\omega}'.
\end{equation}
The differential kernel $\sigma_S \ge 0$ redistributes particles from
$(E',\vec{\omega}')$ to $(E,\vec{\omega})$. The quantity $\sigma_T(E)$
is taken as given so that the loss term $\sigma_T(E)\psi$ controls
$\mathcal{K}$ in the energy norm introduced later.

The kernel $\sigma_S$ contains distinct physical mechanisms. We write
\begin{equation}
  \label{eq:kernel-decomp}
  \sigma_S(\vec{\omega},\vec{\omega}',E' \to E)
  =
  \sigma_{\text{el}}(\vec{\omega},\vec{\omega}',E) \delta(E-E')
  +
  \sigma_{\text{ion}}(\vec{\omega},\vec{\omega}',E' \to E)
  +
  \sigma_{\text{non}}(\vec{\omega},\vec{\omega}',E' \to E),
\end{equation}
where $\sigma_{\text{el}}$ models elastic Coulomb scattering, which is
energy-conserving, $\sigma_{\text{ion}}$ models ionisation losses and
$\sigma_{\text{non}}$ models genuinely nonelastic interactions that can
involve larger energy transfers.

In the energy regime of interest here, ionisation energy loss arises
from a large number of inelastic collisions along each particle track,
with each collision typically transferring only a small amount of
energy. In this setting it is standard to replace the detailed
ionisation energy-loss redistribution by a continuous slowing-down
approximation, in which the cumulative effect of ionisation is
represented by an energy advection term with stopping power $S(E)$.
Neglecting higher-order energy-diffusion (straggling) corrections, this
leads to the charged-particle transport model
\begin{equation}
  \label{eq:BTE}
  \vec{\omega}\cdot\nabla_x \psi(x,E,\vec{\omega})
  - \partial_E\big(S(E)\psi(x,E,\vec{\omega})\big)
  + \sigma_T(E)\psi(x,E,\vec{\omega})
  =
  \mathcal{K}[\psi](x,E,\vec{\omega}),
\end{equation}
where, after applying the slowing down approximation, $\mathcal{K}$
collects the remaining gain contributions, including elastic angular
redistribution and any nonelastic energy-changing events retained in
the model.

In regimes where small-angle Coulomb scattering dominates, the elastic
gain can be approximated by angular diffusion on the sphere via the
Laplace--Beltrami operator $\Delta_{\vec{\omega}}$,
\begin{equation}
  \label{eq:FP-approx}
  \int_{\mathbb{S}^{d-1}}
  \sigma_{\text{el}}(\vec{\omega},\vec{\omega}',E)
  \psi(x,E,\vec{\omega}')  \d\vec{\omega}'
  \approx
  \varepsilon(E)\Delta_{\vec{\omega}}\psi(x,E,\vec{\omega}),
\end{equation}
with $\varepsilon(E)\ge 0$ an energy-dependent angular-diffusion
coefficient \cite{ashby2025positivity}.

Throughout the analysis we adopt the following regularity and
positivity conditions, consistent with the variational framework
developed later: $S \in W^{1,\infty}(I)$ with $S(E)>0$ and $S'(E)\le 0$
on $I$, $\sigma_T \in L^{\infty}(I)$ with $\sigma_T(E)\ge 0$, and
$\sigma_S \in L^{\infty}(\mathbb{S}^{d-1} \times \mathbb{S}^{d-1}
\times I \times I)$ with $\sigma_S\ge 0$. These ensure that the gain
operator \eqref{eq:Kop-general} is bounded on the graph space used for
the weak formulation and that the transport part of \eqref{eq:BTE} is
coercive once inflow boundary conditions are imposed.

\subsection{Stopping power and the Bragg--Kleeman approximation}

Under the continuous slowing down approximation, ionisation losses are
represented in \eqref{eq:BTE} by the stopping power $S(E)$. For light
ions such as protons the dominant mechanism is electronic stopping due
to Coulomb interactions with atomic electrons. A standard empirical
model is the Bragg--Kleeman rule: given constants $\alpha>0$ and
$p\in[1,2]$,
\begin{equation}\label{eq:BraggKleeman}
  S(E)  =  \frac{1}{\alpha p} E^{1-p}.
\end{equation}
This follows from the range--energy relation $R(E)=\alpha E^p$ and is
well validated in the therapeutic window for protons. In particular,
on any admissible interval $I=[E_{\min},E_{\max}] \subset (0,\infty)$
one has $S\in W^{1,\infty}(I)$, $S(E)>0$ and $S'(E)\le 0$, properties
used below to obtain coercivity of the transport operator. One can see
in Figure \ref{fig:proton_fits} how this ansatz is very accurate for
clinically relevant proton energies since ionisation dominates
nonelastic processes.

A more physically motivated alternative is the Bethe--Bloch formula for
the electronic stopping power, which models the mean ionisation loss in
terms of projectile speed and material parameters. It provides a
standard reference across ion species and, with suitable effective
parameters, can be used as an accurate surrogate for electronic
stopping in the therapeutic energy range. In the present work we do not
pursue a full Bethe--Bloch parameterisation and instead adopt the
Bragg--Kleeman approximation \eqref{eq:BraggKleeman} as the simplest
closed-form model that is sufficient for our analysis and numerical
examples.

\begin{figure}[h!]
  \centering
  \pgfmathsetmacro{\EminStop}{0.01}
  \pgfmathsetmacro{\EmaxStop}{150}
  \pgfmathsetmacro{\alphaBK}{2.147e-3}
  \pgfmathsetmacro{\pBK}{1.777}

  \definecolor{bkorange}{RGB}{230,126,34}

  \pgfplotstableread[col sep=comma]{proton_range_data.csv}\protonrangedata
  \pgfplotstableread[col sep=comma]{proton_stopping_bk.csv}\protonstopbk
  \pgfplotstableread[col sep=comma]{proton_stopping_bethe.csv}\protonstopbb

  \begin{subfigure}{.49\linewidth}
    \begin{tikzpicture}[scale=0.93]
      \begin{axis}[
        width=\linewidth,
        grid=both,
        major grid style={black!50},
        xlabel={Energy (MeV)},
        ylabel={Range (cm)},
        legend style={at={(0.02,0.98)},anchor=north west},
      ]

        \addplot[
          only marks,
          mark=square*,
          mark options={scale=1, solid},
          color={black!100}
        ]
        table[x=Energy_MeV, y=Range_cm]{\protonrangedata};
        \addlegendentry{PSTAR range data}

        \addplot[
          dashed,
          color=bkorange,
          line width=1.0,
          domain=0.001:150,
          samples=300
        ]{\alphaBK * x^(\pBK)};
        \addlegendentry{Bragg--Kleeman fit}

      \end{axis}
    \end{tikzpicture}
  \end{subfigure}
  \begin{subfigure}{.49\linewidth}
    \begin{tikzpicture}[scale=0.93]
      \begin{axis}[
        width=\linewidth,
        xmode=log,
        ymode=log,
        grid=both,
        major grid style={black!50},
        xlabel={Energy (MeV)},
        ylabel={Stopping power (MeV cm$^{-1}$)},
        legend style={at={(0.98,0.98)},anchor=north east},
        xmin=\EminStop, xmax=\EmaxStop,
      ]

        \addplot[
          solid,
          color={black!100},
          line width=1.0
        ]
        table[x=Energy_MeV, y=S_hat]{\protonstopbk};
        \addlegendentry{$\widehat S(E)$}

        \addplot[
          dashed,
          color=bkorange,
          line width=1.0
        ]
        table[x=Energy_MeV, y=S_BK]{\protonstopbk};
        \addlegendentry{$S_{\mathrm{BK}}(E)$}

        \addplot[
          densely dashdotted,
          color=blue,
          line width=1.0
        ]
        table[x=Energy_MeV, y=S_BB]{\protonstopbb};
        \addlegendentry{$S_{\mathrm{BB}}(E)$}

      \end{axis}
    \end{tikzpicture}
  \end{subfigure}

  \caption{Left: proton range--energy data in water from PSTAR with
    Bragg--Kleeman fit $R_{\mathrm{BK}}(E)=\alpha E^p$. Right:
    stopping power reconstructed from the same range data via
    $\widehat S(E)=(\mathrm dR/\mathrm dE)^{-1}$, obtained by
    differentiating a cubic spline fit of $R(E)$, together with the
    Bragg--Kleeman stopping power $S_{\mathrm{BK}}(E)=\big(\alpha p
    E^{p-1}\big)^{-1}$ and a Bethe--Bloch stopping power
    $S_{\mathrm{BB}}(E)$. Data taken from PSTAR \cite{berger2005star}.}
  \label{fig:proton_fits}
\end{figure}
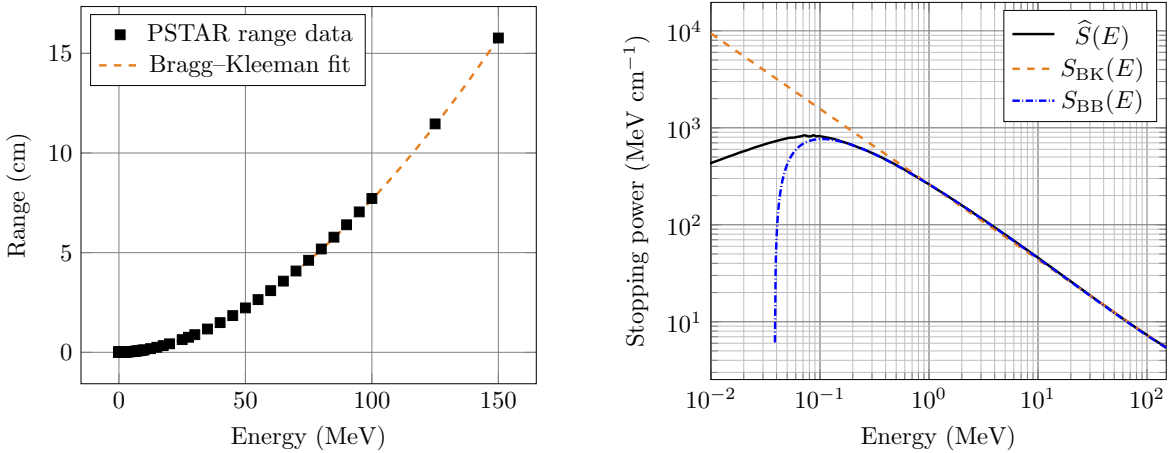

\subsection{Data-informed stopping for heavy ions}

For heavier ions, nuclear and other nonelastic processes become
non-negligible as energy decreases, leading to systematic deviations
from the Bragg--Kleeman form \eqref{eq:BraggKleeman}. These deviations
are not remedied by changing the electronic stopping model alone (for
example by switching to a Bethe--Bloch form), since they reflect
genuinely nonelastic contributions to the total stopping in the
slowing-down regime.  As can be seen in Figure
\ref{fig:carbon_stopping} for carbon ions this effect is evident at
low energies, reflecting fragmentation and secondary production. In
this regime we take $S(E)$ to be a data-informed stopping
power obtained directly from range--energy and stopping data, rather
than enforcing the parametric form \eqref{eq:BraggKleeman}.

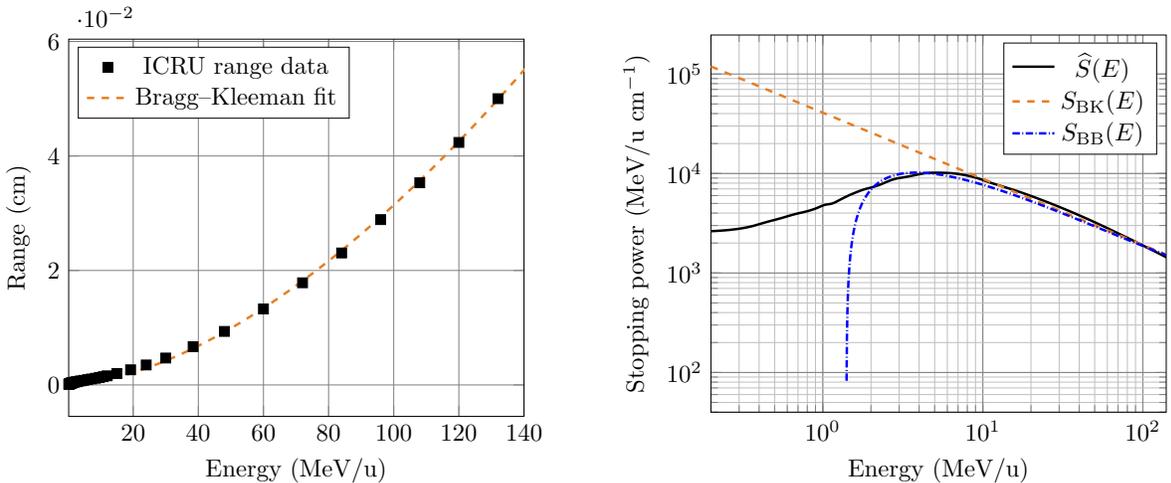
\begin{figure}[h!]
  \centering

  \pgfmathsetmacro{\EminRange}{0.2}
  \pgfmathsetmacro{\EmaxRange}{140}
  \pgfmathsetmacro{\EminStop}{0.2}
  \pgfmathsetmacro{\EmaxStop}{140}

  \pgfmathsetmacro{\alphaBK}{1.470699*10^(-5)}
  \pgfmathsetmacro{\pBK}{1.664787}

  \definecolor{bkorange}{RGB}{230,126,34}

  \pgfplotstableread[col sep=comma]{carbon_range_data.csv}\carbonrangedata
  \pgfplotstableread[col sep=comma]{carbon_stopping_bk.csv}\carbonstopbk
  \pgfplotstableread[col sep=comma]{carbon_stopping_bethe.csv}\carbonstopbb

  \begin{subfigure}{.49\linewidth}
    \begin{tikzpicture}[scale=0.93]
      \begin{axis}[
        width=\linewidth,
        grid=both,
        major grid style={black!50},
        xlabel={Energy (MeV/u)},
        ylabel={Range (cm)},
        legend style={at={(0.02,0.98)},anchor=north west},
        xmin=\EminRange, xmax=\EmaxRange,
        use fpu=true,
      ]

        \addplot[
          only marks,
          mark=square*,
          mark options={scale=1, solid},
          color={black!100}
        ]
        table[x=Energy_MeV, y=Range_cm]{\carbonrangedata};
        \addlegendentry{ICRU range data}

\pgfplotstableread[col sep=comma]{carbon_range_fit.csv}\carbonrangefit

\addplot[
  dashed,
  color=bkorange,
  line width=1.0
]
table[x=Energy_MeV, y=R_BK]{\carbonrangefit};
\addlegendentry{Bragg--Kleeman fit}

      \end{axis}
    \end{tikzpicture}
  \end{subfigure}
  \begin{subfigure}{.49\linewidth}
    \begin{tikzpicture}[scale=0.93]
      \begin{axis}[
        width=\linewidth,
        xmode=log,
        ymode=log,
        grid=both,
        major grid style={black!50},
        xlabel={Energy (MeV/u)},
        ylabel={Stopping power (MeV/u cm$^{-1}$)},
        legend style={at={(0.98,0.98)},anchor=north east},
        xmin=\EminStop, xmax=\EmaxStop,
      ]

        \addplot[
          solid,
          color={black!100},
          line width=1.0
        ]
        table[x=Energy_MeV, y=S_hat]{\carbonstopbk};
        \addlegendentry{$\widehat S(E)$}

        \addplot[
          dashed,
          color=bkorange,
          line width=1.0
        ]
        table[x=Energy_MeV, y=S_BK]{\carbonstopbk};
        \addlegendentry{$S_{\mathrm{BK}}(E)$}

        \addplot[
          densely dashdotted,
          color=blue,
          line width=1.0
        ]
        table[x=Energy_MeV, y=S_BB]{\carbonstopbb};
        \addlegendentry{$S_{\mathrm{BB}}(E)$}

      \end{axis}
    \end{tikzpicture}
  \end{subfigure}

  \caption{Left: carbon range--energy data from ICRU with
    Bragg--Kleeman fit $R_{\mathrm{BK}}(E)=\alpha E^p$ (energy in
    MeV/u). Right: stopping power reconstructed from the same range
    data via $\widehat S(E)=(\mathrm dR/\mathrm dE)^{-1}$, obtained by
    differentiating a cubic spline fit of $R(E)$, together with the
    Bragg--Kleeman stopping power $S_{\mathrm{BK}}(E)=\big(\alpha p
    E^{p-1}\big)^{-1}$ and a Bethe--Bloch-shaped surrogate
    $S_{\mathrm{BB}}(E)$. Data taken from ICRU
    \cite{icru2005report73}.}
  \label{fig:carbon_stopping}
\end{figure}

For the analysis we work on an interval $I=[E_{\min},E_{\max}]$ bounded
away from $0$ and assume that the fitted $S$ satisfies
$S\in W^{1,\infty}(I)$, $S(E)>0$ and $S'(E)\le 0$ on $I$. In practice
this is achieved by using a smooth monotone fit of the tabulated data
over the energies of interest.

\subsection{Boundary conditions}

Since $S(E)>0$ on $I$, the energy drift term $-\partial_E(S\psi)$ induces
a one-way characteristic flow from $E_{\max}$ down to $E_{\min}$. It is
therefore natural to prescribe data both on the spatial inflow boundary
and on the top-energy boundary.

Let $\partial(D\times I)=(\partial D\times I)\cup(D\times\partial I)$ and
write the outward unit normal on $\partial(D\times I)$ as
$\vec n_{(\vec x,E)}:=(\vec n_{\vec x},n_E)$, where $\vec n_{\vec x}$ is
the outward unit normal on $\partial D$, and $n_E=+1$ on $E=E_{\max}$ and
$n_E=-1$ on $E=E_{\min}$. For each fixed $\vec\omega\in\mathbb S^{d-1}$ we
define the inflow and outflow boundaries in $(\vec x,E)$ by
\[
  \Gamma_\pm(\vec\omega)
  :=
  \Big\{(\vec x,E)\in\partial(D\times I) : 
  \pm\big(\vec\omega\cdot\vec n_{\vec x}-S(E) n_E\big)>0\Big\},
\]
and the corresponding phase-space boundary sets by
\[
  \Gamma_\pm
  :=
  \big\{(\vec x,E,\vec\omega)\in \partial(D\times I)\times\mathbb S^{d-1} : 
  (\vec x,E)\in \Gamma_\pm(\vec\omega)\big\}.
\]

This recovers the usual decomposition into spatial and energy inflow/outflow.
On the spatial boundary $\partial D\times I$ one has $n_E=0$, so
\[
  \Gamma_-^{\vec x}
  :=
  \{(\vec x,E,\vec\omega)\in(\partial D\times I)\times\mathbb S^{d-1}:\ \vec\omega\cdot\vec n_{\vec x}<0\},
  \qquad
  \Gamma_+^{\vec x}
  :=
  \{(\vec x,E,\vec\omega)\in(\partial D\times I)\times\mathbb S^{d-1}:\ \vec\omega\cdot\vec n_{\vec x}>0\}.
\]
On the energy boundary $D\times\partial I$ one has $\vec n_{\vec x}=\vec 0$,
so the flux associated with $-\partial_E(S\psi)$ is $-S(E) n_E \psi$.
Since $S(E)>0$, the top boundary $E=E_{\max}$ is inflow and the bottom
boundary $E=E_{\min}$ is outflow. We set
\[
  \Gamma_-^{E} := D\times\{E_{\max}\}\times\mathbb S^{d-1},
  \qquad
  \Gamma_+^{E} := D\times\{E_{\min}\}\times\mathbb S^{d-1},
\]
so that
\[
  \Gamma_- = \Gamma_-^{\vec x}\cup \Gamma_-^{E},
  \qquad
  \Gamma_+ = \Gamma_+^{\vec x}\cup \Gamma_+^{E}.
\]

On $\Gamma_-$ we prescribe inflow data
\begin{equation}\label{eq:inflowBC}
  \psi = g \quad \text{on }\Gamma_-.
\end{equation}
We assume $g$ is square-integrable with respect to the natural flux
weights induced by the boundary term $\vec\omega\cdot\vec n_{\vec x}-S(E)n_E$,
namely $|\vec\omega\cdot\vec n_{\vec x}|$ on $\Gamma_-^{\vec x}$ and
$S(E_{\max})$ on $\Gamma_-^{E}$:
\[
  g\in L^2(\Gamma_-^{\vec x}, |\vec\omega\cdot \vec n_{\vec x}| \d\gamma_{\vec x} \d E \d\vec\omega)
  \ \cap\
  L^2(\Gamma_-^{E}, S(E_{\max}) \d\vec x \d\vec\omega).
\]
No boundary condition is imposed on $\Gamma_+$.

\section{Variational Formulation and Functional Setting}
\label{sec:variational}

To analyse well-posedness and enable discretisation, we adopt a
variational formulation. The first-order transport and energy-advection
terms, together with the non-symmetric gain operator, motivate a graph
space that captures directional and energy drift.

We work in the graph space
\begin{equation}
  \label{eq:Uspace}
  \mathcal{U}
  :=
  \Big\{ u \in L^2(\Omega)  : 
  \vec{\omega}\cdot\nabla_x u \in L^2(\Omega),
  \ \partial_E\qp{S(E)u} \in L^2(\Omega) \Big\}.
\end{equation}

For each fixed $\vec{\omega} \in \mathbb{S}^{d-1}$, define the
directional bilinear form
\begin{multline}
  \label{eq:aform}
  a(u,v;\vec{\omega})
  :=
  \int_D \int_I
  \Big[
    \qp{\vec{\omega}\cdot\nabla_x u}  v
    -
    \partial_E\qp{S(E)u}  v
    +
    \sigma_T(E) u v
  \Big] \d E \d\vec{x}
  \\
  + \frac{1}{2}
  \int_{\Gamma_+(\vec{\omega})}
  \qp{\vec{\omega}\cdot\vec{n}_x - S(E) n_E} u v  \d s,
\end{multline}
and the scattering bilinear form
\begin{equation}
  \label{eq:sform}
  s(u,v;\vec{\omega})
  :=
  \int_D \int_I
  v\qp{\vec{x},E,\vec{\omega}}
  \int_{\mathbb{S}^{d-1}} \int_I
  \sigma_S\qp{\vec{\omega},\vec{\omega}',E' \to E} 
  u\qp{\vec{x},E',\vec{\omega}'}  \d E' \d\vec{\omega}'
   \d E \d\vec{x}.
\end{equation}
The linear form is
\begin{equation}
  \label{eq:lform}
  l(v;\vec{\omega})
  :=
  -\frac{1}{2}
  \int_{\Gamma_-(\vec{\omega})}
  \qp{\vec{\omega}\cdot\vec{n}_x - S(E) n_E} 
  g\qp{\vec{x},E,\vec{\omega}}  v\qp{\vec{x},E,\vec{\omega}}  \d s
  +
  \int_D \int_I
  f\qp{\vec{x},E,\vec{\omega}}  v\qp{\vec{x},E,\vec{\omega}}
   \d E \d\vec{x},
\end{equation}
where $g$ is the inflow boundary datum on $\Gamma_-(\vec{\omega})$ and
$f \in L^2(\Omega)$ is a given volumetric source. In the present paper
$f$ will be used to represent any modelling corrections or additional
physics that are treated explicitly rather than embedded in the drift.

The weak problem reads: find $u \in \mathcal{U}$ such that
\begin{equation}
  \label{eq:weak}
  \int_{\mathbb{S}^{d-1}} a(u,v;\vec{\omega}) \d\vec{\omega}
  =
  \int_{\mathbb{S}^{d-1}}
  \Big[ s(u,v;\vec{\omega}) + l(v;\vec{\omega}) \Big] \d\vec{\omega}
  \qquad \text{for all } v \in \mathcal{U}.
\end{equation}
In the stability and contraction arguments below we apply the coercivity
estimate to functions with homogeneous inflow trace on $\Gamma_-$ (in
particular to source-iteration differences, for which the inflow datum
cancels since $g$ is fixed in $l$).

\begin{theorem}[Coercivity and boundedness of the continuous bilinear forms]
\label{the:coercivity}
Assume $S\in W^{1,\infty}(I)$ with $S(E)>0$ and $S'(E)\le 0$ on $I$,
$\sigma_T\in L^\infty(I)$ with $\sigma_T(E)\ge 0$, and $\sigma_S\in
L^\infty(\mathbb{S}^{d-1}\times\mathbb{S}^{d-1}\times I\times I)$ with
$\sigma_S\ge 0$.

Define the coercivity norm
\begin{equation*}
  \|u\|_{\mathcal{U}}^2
  :=
  \|\sigma_T^{1/2}u\|_{L^2(\Omega)}^2
  + \|(-S'(E))^{1/2}u\|_{L^2(\Omega)}^2
  + \int_{\Gamma_+} \qp{\vec{\omega}\cdot\vec{n}_{\vec x} - S(E) n_E} u^2 \d s,
\end{equation*}
and the stronger norm
\begin{equation*}
  \|u\|_{\mathcal{U}^*}^2
  :=
  \|u\|_{\mathcal{U}}^2
  + \|\nabla_{\vec x} u\|_{L^2(\Omega)}^2
  + \|\partial_E u\|_{L^2(\Omega)}^2.
\end{equation*}
Assume moreover that
\begin{equation*}
  \mu := \essinf_{E\in I}\qp{-S'(E)} > 0,
\end{equation*}
so that $\|w\|_{L^2(\Omega)} \le
\mu^{-1/2}\|(-S')^{1/2}w\|_{L^2(\Omega)} \le
\mu^{-1/2}\|w\|_{\mathcal{U}}$.

Then the bilinear forms $a(u,v;\vec{\omega})$ and $s(u,v;\vec{\omega})$
defined in Section~\ref{sec:variational} satisfy the following:

\begin{enumerate}
\item
  Coercivity. If $u=0$ on $\Gamma_-$, then
  \begin{equation*}
    \int_{\mathbb{S}^{d-1}} a(u,u;\vec{\omega}) \d\vec{\omega}
    \ge
    \tfrac{1}{2}\|u\|_{\mathcal{U}}^2.
  \end{equation*}
  
\item Boundedness. There exists $C_a>0$, depending only on
  $\|S\|_{W^{1,\infty}(I)}$, $\|\sigma_T\|_{L^\infty(I)}$ and $\mu$, such
  that for all $u\in\mathcal{U}^*$ and all $v\in\mathcal{U}$,
  \begin{equation*}
    \Big|\int_{\mathbb{S}^{d-1}} a(u,v;\vec{\omega}) \d\vec{\omega}\Big|
    \le
    C_a \|u\|_{\mathcal{U}^*}\|v\|_{\mathcal{U}}.
  \end{equation*}
  Moreover there exists $C_s>0$, depending only on
  $\|\sigma_S\|_{L^\infty}$ and the measures of $I$ and $\mathbb{S}^{d-1}$,
  such that
  \begin{equation*}
    \Big|\int_{\mathbb{S}^{d-1}} s(u,v;\vec{\omega}) \d\vec{\omega}\Big|
    \le
    C_s \|u\|_{L^2(\Omega)}\|v\|_{L^2(\Omega)}
    \le
    C_s \mu^{-1}\|u\|_{\mathcal{U}}\|v\|_{\mathcal{U}}.
  \end{equation*}
\end{enumerate}

\end{theorem}

\begin{proof}
Fix $\vec{\omega}\in\mathbb{S}^{d-1}$ and abbreviate
$a(u,u):=a(u,u;\vec{\omega})$. By definition,
\begin{align*}
  a(u,u)
  &=
  \int_D \int_I
  \Big[\qp{\vec{\omega}\cdot\nabla_{\vec x} u} u
  - \partial_E\qp{S u} u
  + \sigma_T u^2\Big]\d E \d\vec{x}
  + \tfrac{1}{2}\int_{\Gamma_+(\vec\omega)}\qp{\vec{\omega}\cdot\vec{n}_{\vec x} - S n_E} u^2 \d\gamma.
\end{align*}

We first integrate by parts in $\vec{x}$. Since
$\qp{\vec{\omega}\cdot\nabla_{\vec x} u} u = \tfrac{1}{2}\vec{\omega}\cdot\nabla_{\vec x}\qp{u^2}$,
\begin{equation*}
  \int_D \int_I \qp{\vec{\omega}\cdot\nabla_{\vec x} u} u \d E \d\vec{x}
  =
  \tfrac{1}{2}\int_{\partial D\times I}\qp{\vec{\omega}\cdot\vec{n}_{\vec x}} u^2 \d\gamma.
\end{equation*}
For the energy-advection term, write
\begin{equation*}
  -\partial_E\qp{S u} u
  =
  -\tfrac{1}{2}\partial_E\qp{S u^2} + \tfrac{1}{2}\qp{-S'} u^2,
\end{equation*}
so that integration by parts in $E$ gives
\begin{equation*}
  -\int_D \int_I \partial_E\qp{S u} u \d E \d\vec{x}
  =
  -\tfrac{1}{2}\int_{D\times\partial I} S n_E u^2 \d\gamma
  + \tfrac{1}{2}\int_D \int_I \qp{-S'(E)} u^2 \d E \d\vec{x}.
\end{equation*}
Combining the $\vec{x}$ and $E$ boundary contributions yields
\begin{equation*}
  \tfrac{1}{2}\int_{\partial(D\times I)}
  \qp{\vec{\omega}\cdot\vec{n}_{\vec x} - S n_E} u^2 \d\gamma
  =
  \tfrac{1}{2}\int_{\Gamma_+(\vec\omega)}\qp{\vec{\omega}\cdot\vec{n}_{\vec x} - S n_E} u^2 \d\gamma
  + \tfrac{1}{2}\int_{\Gamma_-(\vec\omega)}\qp{\vec{\omega}\cdot\vec{n}_{\vec x} - S n_E} u^2 \d\gamma.
\end{equation*}
Therefore
\begin{align*}
  a(u,u)
  &=
  \|\sigma_T^{1/2}u(\cdot,\cdot,\vec\omega)\|_{L^2(D\times I)}^2
  + \tfrac{1}{2}\|(-S')^{1/2}u(\cdot,\cdot,\vec\omega)\|_{L^2(D\times I)}^2
  \\
  &\qquad + \int_{\Gamma_+(\vec\omega)}\qp{\vec{\omega}\cdot\vec{n}_{\vec x} - S n_E} u^2 \d\gamma
  + \tfrac{1}{2}\int_{\Gamma_-(\vec\omega)}\qp{\vec{\omega}\cdot\vec{n}_{\vec x} - S n_E} u^2 \d\gamma.
\end{align*}
If $u=0$ on $\Gamma_-$, then for each fixed $\vec\omega$ we have $u=0$
on $\Gamma_-(\vec\omega)$ and the last term vanishes. Hence
\begin{equation*}
  a(u,u)
  \ge
  \tfrac{1}{2}\|\sigma_T^{1/2}u(\cdot,\cdot,\vec\omega)\|_{L^2(D\times I)}^2
  + \tfrac{1}{2}\|(-S')^{1/2}u(\cdot,\cdot,\vec\omega)\|_{L^2(D\times I)}^2
  + \tfrac{1}{2}\int_{\Gamma_+(\vec\omega)}\qp{\vec{\omega}\cdot\vec{n}_{\vec x} - S n_E} u^2 \d\gamma.
\end{equation*}
Integrating this inequality over $\mathbb{S}^{d-1}$ and using $\d
s=\d\gamma \d\vec\omega$ gives
\begin{equation*}
  \int_{\mathbb{S}^{d-1}} a(u,u;\vec{\omega}) \d\vec{\omega}
  \ge
  \tfrac{1}{2}\|\sigma_T^{1/2}u\|_{L^2(\Omega)}^2
  + \tfrac{1}{2}\|(-S')^{1/2}u\|_{L^2(\Omega)}^2
  + \tfrac{1}{2}\int_{\Gamma_+}\qp{\vec{\omega}\cdot\vec{n}_{\vec x} - S n_E} u^2 \d s
  =
  \tfrac{1}{2}\|u\|_{\mathcal{U}}^2,
\end{equation*}
which proves coercivity.

For boundedness of $a$, fix $\vec{\omega}$ and apply Cauchy--Schwarz,
together with
\[
\|\partial_E\qp{S u}\|_{L^2(\Omega)}\le \|S\|_{L^\infty(I)}\|\partial_E u\|_{L^2(\Omega)} + \|S'\|_{L^\infty(I)}\|u\|_{L^2(\Omega)}
\]
we have
\begin{align*}
  \Big|\int_D\int_I \qp{\vec{\omega}\cdot\nabla_{\vec x} u} v\Big|
  &\le
  \|\nabla_{\vec x} u(\cdot,\cdot,\vec\omega)\|_{L^2(D\times I)}\|v(\cdot,\cdot,\vec\omega)\|_{L^2(D\times I)},\\
  \Big|\int_D\int_I \partial_E\qp{S u} v\Big|
  &\le
  \|\partial_E\qp{S u}(\cdot,\cdot,\vec\omega)\|_{L^2(D\times I)}\|v(\cdot,\cdot,\vec\omega)\|_{L^2(D\times I)}\\
  &\le
  \|S\|_{W^{1,\infty}(I)}\qp{\|\partial_E u(\cdot,\cdot,\vec\omega)\|_{L^2(D\times I)}+\|u(\cdot,\cdot,\vec\omega)\|_{L^2(D\times I)}}\|v(\cdot,\cdot,\vec\omega)\|_{L^2(D\times I)},\\
  \Big|\int_D\int_I \sigma_T u v\Big|
  &\le
  \|\sigma_T\|_{L^\infty(I)}\|u(\cdot,\cdot,\vec\omega)\|_{L^2(D\times I)}\|v(\cdot,\cdot,\vec\omega)\|_{L^2(D\times I)},\\
  \Big|\tfrac{1}{2}\int_{\Gamma_+(\vec\omega)}\qp{\vec{\omega}\cdot\vec{n}_{\vec x} - S n_E} u v \d\gamma\Big|
  &\le
  \tfrac{1}{2}
  \Big(\int_{\Gamma_+(\vec\omega)}\qp{\vec{\omega}\cdot\vec{n}_{\vec x} - S n_E} u^2 \d\gamma\Big)^{1/2}
  \Big(\int_{\Gamma_+(\vec\omega)}\qp{\vec{\omega}\cdot\vec{n}_{\vec x} - S n_E} v^2 \d\gamma\Big)^{1/2}.
\end{align*}
Using $\|w\|_{L^2(\Omega)}\le \mu^{-1/2}\|w\|_{\mathcal{U}}$ for $w=v$ and
$\|u\|_{L^2(\Omega)}\le \mu^{-1/2}\|u\|_{\mathcal{U}}\le \mu^{-1/2}\|u\|_{\mathcal{U}^*}$,
we obtain a bound of the form
$|a(u,v;\vec{\omega})|\le C_a\|u\|_{\mathcal{U}^*}\|v\|_{\mathcal{U}}$ with $C_a$ independent of $\vec{\omega}$.
Integrating over $\mathbb{S}^{d-1}$ yields the stated estimate.

For the scattering form, define
\begin{equation*}
  \mathcal{K}[u]\qp{\vec{x},E,\vec{\omega}}
  :=
  \int_{\mathbb{S}^{d-1}}\int_I
  \sigma_S\qp{\vec{\omega},\vec{\omega}',E'\to E} 
  u\qp{\vec{x},E',\vec{\omega}'} \d E' \d\vec{\omega}'.
\end{equation*}
By Fubini and $\|\sigma_S\|_{L^\infty}<\infty$,
\begin{equation*}
  |\mathcal{K}[u]\qp{\vec{x},E,\vec{\omega}}|
  \le
  \|\sigma_S\|_{L^\infty}\int_{\mathbb{S}^{d-1}}\int_I
  |u\qp{\vec{x},E',\vec{\omega}'}| \d E' \d\vec{\omega}'.
\end{equation*}
Applying Cauchy--Schwarz in $(E',\vec{\omega}')$ gives
\[
  |\mathcal{K}[u]\qp{\vec{x},E,\vec{\omega}}|
  \le
  \|\sigma_S\|_{L^\infty} \qp{|I| |\mathbb S^{d-1}|}^{1/2}
  \Big(\int_{\mathbb{S}^{d-1}}\int_I |u\qp{\vec{x},E',\vec{\omega}'}|^2 \d E' \d\vec{\omega}'\Big)^{1/2}.
\]
Squaring, integrating over $(\vec{x},E,\vec{\omega})$ and using Fubini yields
$\|\mathcal{K}[u]\|_{L^2(\Omega)} \le C_s\|u\|_{L^2(\Omega)}$ with $C_s$ depending only on
$\|\sigma_S\|_{L^\infty}$ and the measures of $I$ and $\mathbb{S}^{d-1}$. Hence
\begin{equation*}
  \Big|\int_{\mathbb{S}^{d-1}} s(u,v;\vec{\omega}) \d\vec{\omega}\Big|
  =
  \Big|\int_\Omega v \mathcal{K}[u]\Big|
  \le
  \|v\|_{L^2(\Omega)}\|\mathcal{K}[u]\|_{L^2(\Omega)}
  \le
  C_s\|u\|_{L^2(\Omega)}\|v\|_{L^2(\Omega)}.
\end{equation*}
The final inequality in the theorem follows from $\|w\|_{L^2(\Omega)}\le \mu^{-1/2}\|w\|_{\mathcal{U}}$ for $w=u,v$.
\end{proof}

\subsection{Source iteration at the continuum level}
\label{sec:SI-continuum}

The weak formulation \eqref{eq:weak} couples all directions through the
gain form $s(\cdot,\cdot;\vec{\omega})$. To decouple angles and obtain
an efficient solver, we employ a (Picard) source iteration in which the
gain is lagged. Given an initial guess $u^{(0)}\in\mathcal{U}$, define
$\{u^{(n)}\}_{n\ge 1}$ by the following directional problems: for each
$\vec{\omega}\in\mathbb{S}^{d-1}$, find
$u^{(n)}(\cdot,\cdot,\vec{\omega})\in\mathcal{U}$ such that
\begin{equation}
  \label{eq:SI-weak}
  a\qp{u^{(n)},v;\vec{\omega}}
  =
  s\qp{u^{(n-1)},v;\vec{\omega}}
  +
  l\qp{v;\vec{\omega}}
  \qquad \text{for all } v\in\mathcal{U},
\end{equation}
where $a$, $s$ and $l$ are the forms from Section~\ref{sec:variational}.
In particular, the inflow boundary datum $g$ is enforced in the linear
functional $l(\cdot;\vec{\omega})$ and is not updated with $n$.

\begin{proposition}[Convergence of source iteration]
\label{prop:SI-continuum}
Assume the hypotheses of Theorem~\ref{the:coercivity} and that
the gain operator induces a bounded bilinear form on $\mathcal{U}$ in
the sense that there exists $\eta\ge 0$ such that
\begin{equation}
  \label{eq:K-U-bound}
  \Big|\int_{\mathbb{S}^{d-1}} s\qp{w,v;\vec{\omega}} \d\vec{\omega}\Big|
  \le
  \eta  \|w\|_{\mathcal{U}} \|v\|_{\mathcal{U}}
  \qquad \text{for all } w,v\in\mathcal{U}.
\end{equation}
If $2\eta<1$, then the Picard map induced by \eqref{eq:SI-weak} is a
strict contraction on $\mathcal{U}$ in the $\|\cdot\|_{\mathcal U}$ norm.
Consequently, for any initial guess $u^{(0)}\in\mathcal U$ the iterates
$\{u^{(n)}\}_{n\ge 0}$ defined by \eqref{eq:SI-weak} converge strongly in
$\mathcal{U}$ to the unique solution $u$ of the variational problem
\eqref{eq:weak}.

In particular, by Theorem~\ref{the:coercivity} one may take
$\eta = C_s \mu^{-1}$, where $C_s$ is the scattering boundedness
constant.
\end{proposition}

\begin{proof}
Define the integrated bilinear forms
\[
  A(u,v) := \int_{\mathbb S^{d-1}} a(u,v;\vec\omega) \d\vec\omega,
  \qquad
  S(u,v) := \int_{\mathbb S^{d-1}} s(u,v;\vec\omega) \d\vec\omega,
\]
and likewise $L(v):=\int_{\mathbb S^{d-1}}l(v;\vec\omega) \d\vec\omega$.
Then the weak problem \eqref{eq:weak} reads $A(u,v)=S(u,v)+L(v)$ for all
$v\in\mathcal U$, and the source iteration \eqref{eq:SI-weak} reads
\begin{equation*}
  A\qp{u^{(n)},v} = S\qp{u^{(n-1)},v}+L(v)\qquad\text{for all }v\in\mathcal U.
\end{equation*}

Let $T:\mathcal U\to\mathcal U$ denote the Picard map defined by
$u^{(n)}=T(u^{(n-1)})$, that is, $T(w)$ is the unique element of
$\mathcal U$ satisfying
\begin{equation}
  \label{eq:Tdef}
  A\qp{T(w),v} = S\qp{w,v}+L(v)\qquad\text{for all }v\in\mathcal U.
\end{equation}
(Existence and uniqueness of $T(w)$ follow from the well-posedness
associated with Theorem~\ref{the:coercivity}; in particular, the boundary
datum is fixed in $L$ and is therefore the same for all arguments $w$.)

To prove contraction, let $w_1,w_2\in\mathcal U$ and set
$\delta := T(w_1)-T(w_2)$. Subtracting \eqref{eq:Tdef} for $w_1$ and $w_2$
gives
\begin{equation}
  \label{eq:Tdiff}
  A(\delta,v) = S\qp{w_1-w_2,v}\qquad\text{for all }v\in\mathcal U.
\end{equation}
Since the inflow datum is fixed in $L$, it cancels in the difference, and
$\delta$ has homogeneous inflow trace on $\Gamma_-$ (equivalently,
$\delta$ belongs to the homogeneous-inflow subspace used in
Theorem~\ref{the:coercivity}). Choosing $v=\delta$ in \eqref{eq:Tdiff} and
applying the coercivity estimate from Theorem~\ref{the:coercivity} yields
\begin{equation}
  \label{eq:Tcoerc}
  \tfrac12 \|\delta\|_{\mathcal U}^2 \le \Big|S\qp{w_1-w_2,\delta}\Big|.
\end{equation}
Using \eqref{eq:K-U-bound} gives
\[
  \tfrac12 \|\delta\|_{\mathcal U}^2
  \le
  \eta \|w_1-w_2\|_{\mathcal U} \|\delta\|_{\mathcal U},
  \qquad\text{hence}\qquad
  \|\delta\|_{\mathcal U}\le 2\eta \|w_1-w_2\|_{\mathcal U}.
\]
Therefore $\|T(w_1)-T(w_2)\|_{\mathcal U}\le 2\eta \|w_1-w_2\|_{\mathcal U}$.
If $2\eta<1$, $T$ is a strict contraction on $\mathcal U$. Banach's fixed
point theorem implies that $u^{(n)}\to u$ strongly in $\mathcal U$, where
$u$ is the unique fixed point of $T$, and inserting $u=T(u)$ in
\eqref{eq:Tdef} shows that $u$ satisfies \eqref{eq:weak}. Uniqueness of
the solution of \eqref{eq:weak} follows by applying the same estimate to
the difference of two solutions.

Finally, the choice $\eta=C_s\mu^{-1}$ follows immediately from the
scattering boundedness estimate in Theorem~\ref{the:coercivity}.
\end{proof}

\section{Angular approximation and error control}
\label{sec:angular}

In this section we focus on the angular discretisation of elastic
Coulomb scattering. For simplicity we neglect nonelastic interactions
and retain only the energy-conserving elastic component of the gain.
Accordingly we work with an elastic kernel
$\sigma_{\text{el}}(\vec{\omega},\vec{\omega}',E)$ and the associated
gain
\[
  \mathcal{K}_{\text{el}}[\psi](\vec{x},E,\vec{\omega})
  :=
  \int_{\mathbb{S}^{d-1}}
  \sigma_{\text{el}}(\vec{\omega},\vec{\omega}',E) 
  \psi(\vec{x},E,\vec{\omega}') \d\vec{\omega}'.
\]
The stopping power $S(E)$ is the same as in Section~\ref{sec:models}.

To reduce computational cost, we restrict the angular domain to a
forward cone centred on the primary beam direction. Let
$\vec{\omega}_\star\in\mathbb{S}^{d-1}$ be a reference direction and,
for a half-angle $\theta_c>0$, define
\begin{equation}
  \label{eq:cone}
  \mathcal{C}(\vec{\omega}_\star,\theta_c)
  :=
  \Big\{ \vec{\omega}'\in\mathbb{S}^{d-1} :
  \arccos\qp{\vec{\omega}_\star \cdot \vec{\omega}'}\le \theta_c \Big\},
\end{equation}
illustrated in Figure \ref{fig:cone}. In forward-peaked regimes (high
energies, weak multiple scattering) one may take $\theta_c$ small,
whereas stronger scattering requires a larger cone.

We therefore introduce the cone-truncated gain operator
\[
  \mathcal{K}_{\text{el}}^{c}[\psi](\vec{x},E,\vec{\omega})
  :=
  \int_{\mathcal{C}(\vec{\omega}_\star,\theta_c)}
  \sigma_{\text{el}}(\vec{\omega},\vec{\omega}',E) 
  \psi(\vec{x},E,\vec{\omega}') \d\vec{\omega}'.
\]
The difference $\mathcal{K}_{\text{el}}[\psi]-\mathcal{K}_{\text{el}}^{c}[\psi]$
is the cone-truncation remainder, and it will be treated as a distinct
error contribution below.

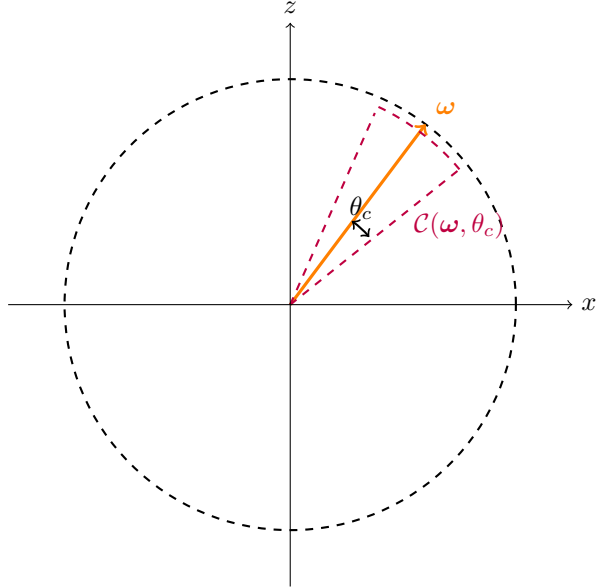
\begin{figure}[h!]
  \centering
  \begin{tikzpicture}[scale=1.5]

\draw[thick, dashed] (0,0) circle (2);

\draw[->, very thick, orange] (0,0) -- (1.2,1.6) node[above right] {$\vec{\w}$};

\draw[thick, dashed, purple] (0,0) -- (1.5,1.2);
\draw[thick, dashed, purple] (0,0) -- (0.75,1.7);
\draw[thick, dashed, purple] (1.5,1.2) arc[start angle=40, end angle=65, radius=2.12];

\draw[<->, thick] (0.7,0.6) arc[start angle=40, end angle=55, radius=0.8] node[midway, above] {$\theta_c$};
\node[purple] at (1.5,.7) {$\mathcal{C}(\vec{\w}, \theta_c)$};

\draw[->] (-2.5,0) -- (2.5,0) node[right] {$x$};
\draw[->] (0,-2.5) -- (0,2.5) node[above] {$z$};

\end{tikzpicture}
  \caption{Angular computational domain restricted to a cone centred on
    $\vec{\omega}_\star$.}
  \label{fig:cone}
\end{figure}

\subsection{Angular discretisation}

Let $\{\vec{\omega}_j\}_{j=1}^Q\subset\mathcal{C}(\vec{\omega}_\star,\theta_c)$
be quadrature nodes with weights $\{\mu_j\}_{j=1}^Q>0$. For sufficiently
regular $f$,
\begin{equation}
  \label{eq:cone-quadrature}
  \int_{\mathcal{C}(\vec{\omega}_\star,\theta_c)} f(\vec{\omega}') \d\vec{\omega}'
  \approx
  \sum_{j=1}^Q \mu_j f(\vec{\omega}_j).
\end{equation}
Applying \eqref{eq:cone-quadrature} to the cone-truncated elastic gain
yields
\begin{equation}
  \label{eq:elastic-cone}
  \int_{\mathcal{C}(\vec{\omega}_\star,\theta_c)}
  \sigma_{\text{el}}(\vec{\omega},\vec{\omega}',E) 
  \psi(\vec{x},E,\vec{\omega}') \d\vec{\omega}'
  \approx
  \sum_{j=1}^Q \mu_j 
  \sigma_{\text{el}}(\vec{\omega},\vec{\omega}_j,E) 
  \psi(\vec{x},E,\vec{\omega}_j).
\end{equation}

Introducing the directional unknowns
$\psi_i(\vec{x},E):=\psi(\vec{x},E,\vec{\omega}_i)$ for $i=1,\dots,Q$,
the corresponding discrete-ordinates system for the cone-truncated model reads
\begin{equation}
  \label{eq:DOM}
  \vec{\omega}_i\cdot\nabla_{\vec{x}}\psi_i(\vec{x},E)
  -
  \partial_E\qp{S(E)\psi_i(\vec{x},E)}
  +
  \sigma_T(E)\psi_i(\vec{x},E)
  =
  \sum_{j=1}^Q \mu_j 
  \sigma_{\text{el}}(\vec{\omega}_i,\vec{\omega}_j,E) 
  \psi_j(\vec{x},E),
  \qquad i=1,\dots,Q.
\end{equation}
This is a coupled family of energy-dependent transport equations, one
per direction, capturing forward-peaked angular exchange within
$\mathcal{C}(\vec{\omega}_\star,\theta_c)$. For each ordinate $i$, the
inflow boundary condition is imposed on $\Gamma_-(\vec\omega_i)$ in the
sense of Section~\ref{sec:models}, namely on the spatial inflow set
$\{\vec\omega_i\cdot \vec n_{\vec x}<0\}$ together with the top-energy
boundary $E=E_{\max}$, with no condition on the corresponding outflow.

\begin{remark}[Continuum conservation]
Assume the elastic kernel satisfies the mass-conserving normalisation
\[
  \int_{\mathbb{S}^{d-1}} \sigma_{\text{el}}(\vec{\omega},\vec{\omega}',E) \d\vec{\omega}
  =
  \Sigma_{\text{el}}(E)
  \qquad \text{for a.e. } E\in I,
\]
where $\Sigma_{\text{el}}(E)$ is independent of $\vec{\omega}'$. Then the
elastic gain satisfies the identity
\[
  \int_{\mathbb{S}^{d-1}}
  \Big(\int_{\mathbb{S}^{d-1}}
  \sigma_{\text{el}}(\vec{\omega},\vec{\omega}',E) 
  \psi(\vec{x},E,\vec{\omega}') \d\vec{\omega}'\Big)\d\vec{\omega}
  =
  \Sigma_{\text{el}}(E)
  \int_{\mathbb{S}^{d-1}}
  \psi(\vec{x},E,\vec{\omega}) \d\vec{\omega}.
\]
Consequently, in the purely elastic setting where the loss term contains
the elastic removal $\Sigma_{\text{el}}(E)\psi$ (for example
$\sigma_T(E)=\Sigma_{\text{el}}(E)$ when no other interactions are
present), the net elastic scattering contribution conserves particle
number at each energy after integrating over $\mathbb{S}^{d-1}$.

If the angular domain is truncated to a cone
$\mathcal{C}(\vec{\omega}_\star,\theta_c)$, then the corresponding
cone-truncated gain $\mathcal{K}_{\mathrm{el}}^{c}$ fails to satisfy the
full-sphere identity only through scattering events that transfer mass to
$\mathbb{S}^{d-1}\setminus\mathcal{C}(\vec{\omega}_\star,\theta_c)$.
Quantitative bounds are given later via cone-truncation error estimates.
\end{remark}

\begin{proposition}[Discrete conservation under exact angular quadrature]
\label{prop:conservation}
Assume the elastic kernel is rotationally invariant at fixed energy,
$\sigma_{\mathrm{el}}(\vec{\omega},\vec{\omega}',E)
=\kappa_E(\vec{\omega}\cdot\vec{\omega}')$ with
$\kappa_E:[-1,1]\to\mathbb{R}_+$. Fix an integer $L\ge 0$ and assume that,
for a.e.\ $E\in I$, $\kappa_E$ admits a truncated Legendre expansion of
degree at most $L$, that is,
\[
  \kappa_E(\mu) = \sum_{\ell=0}^{L} a_\ell(E) P_\ell(\mu)
  \qquad\text{for }\mu\in[-1,1],
\]
with coefficients $a_\ell(E)\ge 0$. Let
$\{(\vec{\omega}_i,\varpi_i)\}_{i=1}^Q\subset
\mathbb{S}^{d-1}\times(0,\infty)$ be a full-sphere quadrature rule
that is exact for spherical harmonics (up to degree $L$), equivalently
exact for all restrictions to $\mathbb S^{d-1}$ of polynomials in
$\vec\omega$ of total degree at most $L$. In particular, for each
$j=1,\dots,Q$ and a.e.\ $E\in I$, this implies exactness for the
corresponding zonal functions $\vec{\omega}\mapsto
P_\ell(\vec{\omega}\cdot\vec{\omega}_j)$, $\ell=0,\dots,L$, and hence
for $\vec{\omega}\mapsto
\kappa_E(\vec{\omega}\cdot\vec{\omega}_j)$. Define
\[
  \Sigma_{\mathrm{el}}(E)
  :=
  \int_{\mathbb{S}^{d-1}}
  \sigma_{\mathrm{el}}(\vec{\omega},\vec{\omega}',E) \d\vec{\omega},
\]
which is independent of $\vec{\omega}'$ by rotational invariance. Then
for any values $\{\psi_j(\vec{x},E)\}_{j=1}^Q$ one has the discrete balance
\[
  \sum_{i=1}^Q \varpi_i
  \Big(
    \sum_{j=1}^Q \mu_j 
    \sigma_{\mathrm{el}}(\vec{\omega}_i,\vec{\omega}_j,E) 
    \psi_j(\vec{x},E)
    -
    \Sigma_{\mathrm{el}}(E) \psi_i(\vec{x},E)
  \Big)
  = 0,
\]
where $\psi_i(\vec{x},E):=\psi(\vec{x},E,\vec{\omega}_i)$. Equivalently,
\[
  \sum_{i=1}^Q \varpi_i \sum_{j=1}^Q \mu_j 
  \sigma_{\mathrm{el}}(\vec{\omega}_i,\vec{\omega}_j,E) 
  \psi_j(\vec{x},E)
  =
  \Sigma_{\mathrm{el}}(E)\sum_{i=1}^Q \varpi_i \psi_i(\vec{x},E).
\]
\end{proposition}

\begin{proof}
Using
$\sigma_{\mathrm{el}}(\vec{\omega}_i,\vec{\omega}_j,E)=\kappa_E(\vec{\omega}_i\cdot\vec{\omega}_j)$,
we compute
\[
  \sum_{i=1}^Q \varpi_i \sum_{j=1}^Q \mu_j 
  \sigma_{\mathrm{el}}(\vec{\omega}_i,\vec{\omega}_j,E) \psi_j
  =
  \sum_{j=1}^Q \mu_j \psi_j
  \Big(\sum_{i=1}^Q \varpi_i \kappa_E(\vec{\omega}_i\cdot\vec{\omega}_j)\Big).
\]
By quadrature exactness, for each fixed $j$ we have
\[
  \sum_{i=1}^Q \varpi_i \kappa_E(\vec{\omega}_i\cdot\vec{\omega}_j)
  =
  \int_{\mathbb{S}^{d-1}} \kappa_E(\vec{\omega}\cdot\vec{\omega}_j) \d\vec{\omega}
  =
  \int_{\mathbb{S}^{d-1}} \sigma_{\mathrm{el}}(\vec{\omega},\vec{\omega}_j,E) \d\vec{\omega}
  =
  \Sigma_{\mathrm{el}}(E),
\]
and the right-hand side is independent of $j$. Substituting gives
\[
  \sum_{i=1}^Q \varpi_i \sum_{j=1}^Q \mu_j 
  \sigma_{\mathrm{el}}(\vec{\omega}_i,\vec{\omega}_j,E) \psi_j
  =
  \Sigma_{\mathrm{el}}(E)\sum_{j=1}^Q \mu_j \psi_j,
\]
which is equivalent to the stated balance.
\end{proof}

\begin{remark}
Exact conservation in Proposition~\ref{prop:conservation} relies on a
full-sphere quadrature rule that is exact for the relevant zonal
functions. In the cone-restricted setting
$\mathcal{C}(\vec{\omega}_\star,\theta_c)$ used for computation, the same
calculation applies to the cone-truncated gain, but conservation holds
only up to a remainder representing scattering events that transfer mass
to $\mathbb{S}^{d-1}\setminus\mathcal{C}(\vec{\omega}_\star,\theta_c)$.
Quantitative bounds are given later via cone-truncation estimates.
\end{remark}

Let $U:=L^2(D\times I)$ and write $\vec{u}=(u_1,\dots,u_Q)\in U^Q$ for
the directional components associated with nodes
$\{\vec{\omega}_i\}_{i=1}^Q$. We work with the per-ordinate graph
spaces
\[
  \mathcal{U}_i
  :=
  \Big\{ w\in L^2(D\times I) : 
  \vec{\omega}_i\cdot\nabla_{\vec x} w \in L^2(D\times I),
  \ \partial_E\qp{S(E)w}\in L^2(D\times I)\Big\},
  \qquad i=1,\dots,Q,
\]
and the product space $\mathcal{U}^Q:=\prod_{i=1}^Q \mathcal{U}_i$.
We write $\d\gamma$ for the surface measure on $\partial(D\times I)$.

The discrete elastic scattering form for the $i$-th direction is
\begin{equation}
  \label{eq:sQ-def}
  s^Q(\vec{u},v_i;\vec{\omega}_i)
  :=
  \int_D \int_I
  v_i(\vec{x},E)
  \sum_{j=1}^Q \mu_j 
  \sigma_{\mathrm{el}}(\vec{\omega}_i,\vec{\omega}_j,E) 
  u_j(\vec{x},E)
   \d E \d\vec{x}.
\end{equation}
Consistently with Section~\ref{sec:variational}, the directional
bilinear form retains the transport and boundary terms,
\begin{multline}
  \label{eq:aQ-def}
  a(u_i,v_i;\vec{\omega}_i)
  :=
  \int_D \int_I
  \Big[
    \qp{\vec{\omega}_i\cdot\nabla_{\vec{x}} u_i} v_i
    -
    \partial_E\qp{S(E)u_i} v_i
    +
    \sigma_T(E) u_i v_i
  \Big] \d E \d\vec{x}
  \\
  + \frac{1}{2}\int_{\Gamma_+(\vec{\omega}_i)}
  \qp{\vec{\omega}_i\cdot\vec{n}_{\vec x} - S(E)n_E} u_i v_i \d\gamma,
\end{multline}
and the inflow linear form for the $i$-th direction is
\begin{equation}
  \label{eq:lQ-def}
  l(v_i;\vec{\omega}_i)
  :=
  -\frac{1}{2}\int_{\Gamma_-(\vec{\omega}_i)}
  \qp{\vec{\omega}_i\cdot\vec{n}_{\vec x} - S(E)n_E} 
  g(\vec{x},E,\vec{\omega}_i) v_i(\vec{x},E) \d\gamma.
\end{equation}
Neglecting nonelastic interactions, the discrete weak problem in angle
is: find $\vec{u}\in\mathcal{U}^Q$ such that, for each $i=1,\dots,Q$,
\begin{equation}
  \label{eq:discrete_in_angle}
  a(u_i,v_i;\vec{\omega}_i)
  =
  s^Q(\vec{u},v_i;\vec{\omega}_i)
  +
  l(v_i;\vec{\omega}_i)
  \qquad \text{for all } v_i\in\mathcal{U}_i.
\end{equation}

\subsection{Source iteration}
\label{subsec:sourceiteration}

The coupled directional system \eqref{eq:DOM} contains angular
scattering terms that redistribute particles across directions. We
utilise a discrete source-iteration scheme, which decouples the
directional equations at each step and permits independent, parallel
directional solves.

Given $\vec{u}^{(0)}=(u^{(0)}_1,\dots,u^{(0)}_Q)\in\mathcal{U}^Q$, define
$\vec{u}^{(n)}=(u^{(n)}_1,\dots,u^{(n)}_Q)$ by, for each $i=1,\dots,Q$,
\begin{equation}
  \label{eq:SI-strong-discrete}
  \vec{\omega}_i\cdot\nabla_{\vec x} u_i^{(n)}(\vec{x},E)
  -
  \partial_E\qp{S(E)u_i^{(n)}(\vec{x},E)}
  +
  \sigma_T(E)u_i^{(n)}(\vec{x},E)
  =
  \sum_{j=1}^Q \mu_j 
  \sigma_{\mathrm{el}}(\vec{\omega}_i,\vec{\omega}_j,E) 
  u_j^{(n-1)}(\vec{x},E),
\end{equation}
with inflow data $u_i^{(n)}=g(\cdot,\cdot,\vec{\omega}_i)$ on
$\Gamma_-(\vec{\omega}_i)$.

The corresponding discrete weak formulation reads: given
$\vec{u}^{(n-1)}\in\mathcal{U}^Q$, find
$\vec{u}^{(n)}\in\mathcal{U}^Q$ such that, for each $i=1,\dots,Q$,
\begin{equation}
  \label{eq:SI-weak-discrete}
  a\qp{u_i^{(n)},v_i;\vec{\omega}_i}
  =
  s^Q\qp{\vec{u}^{(n-1)},v_i;\vec{\omega}_i}
  +
  l\qp{v_i;\vec{\omega}_i}
  \qquad \text{for all } v_i\in\mathcal{U}_i.
\end{equation}

\begin{proposition}[Convergence of discrete-ordinates source iteration]
\label{prop:SI-discrete}
Assume the hypotheses of Theorem~\ref{the:coercivity} and neglect
nonelastic interactions, so that scattering is given by the elastic
kernel $\sigma_{\mathrm{el}}$. Let $\{\vec\omega_i,\varpi_i\}_{i=1}^Q$ be a
(full-sphere or cone-restricted) quadrature rule and consider the
discrete-ordinates weak iteration \eqref{eq:SI-weak-discrete} on the
product graph space $\mathcal U^Q=\prod_{i=1}^Q\mathcal U_i$.

Assume that the discrete scattering form satisfies the boundedness
estimate: there exists $\eta_Q\ge 0$ such that
\begin{equation}
  \label{eq:KQ-U-bound}
  \Big|\sum_{i=1}^Q \varpi_i  s^Q\qp{\vec w,v_i;\vec\omega_i}\Big|
  \le
  \eta_Q \|\vec w\|_{\mathcal U^Q} \|\vec v\|_{\mathcal U^Q}
  \qquad\text{for all }\vec w,\vec v\in\mathcal U^Q,
\end{equation}
where $\|\vec w\|_{\mathcal U^Q}^2:=\sum_{i=1}^Q \varpi_i \|w_i\|_{\mathcal U_i}^2$
and $\|\cdot\|_{\mathcal U_i}$ is the coercivity norm induced by
$a(\cdot,\cdot;\vec\omega_i)$.

If $2\eta_Q<1$, then the discrete source iteration \eqref{eq:SI-weak-discrete}
is a strict contraction on $\mathcal U^Q$ in the norm $\|\cdot\|_{\mathcal U^Q}$.
Consequently, for any initial guess $\vec u^{(0)}\in\mathcal U^Q$ the iterates
$\vec u^{(n)}$ converge strongly in $\mathcal U^Q$ to the unique solution
$\vec u\in\mathcal U^Q$ of the discrete weak problem \eqref{eq:discrete_in_angle}.
\end{proposition}

\begin{proof}
The argument is identical to the continuum proof of
Proposition~\ref{prop:SI-continuum}, with the integrated forms over
$\mathbb S^{d-1}$ replaced by the quadrature sums
$\sum_{i=1}^Q\varpi_i(\cdot)$ and with the space $\mathcal U$ replaced by
the product graph space $\mathcal U^Q$. In particular, taking differences
between successive iterates eliminates the fixed inflow datum, so the
error has homogeneous inflow trace on each $\Gamma_-(\vec\omega_i)$.
Testing the discrete error equation with the discrete error, applying
coercivity of the transport part (the discrete analogue of
Theorem~\ref{the:coercivity} with $\vec\omega$ fixed to $\vec\omega_i$),
and then using \eqref{eq:KQ-U-bound} yields the contraction factor
$2\eta_Q$.
\end{proof}

\section{Error control for the angular approximation}
\label{sec:errorcontrol}

This section quantifies the error introduced by restricting and
discretising angle and by terminating the source iteration. Throughout we
work with the elastic model of Section~\ref{sec:angular}, that is,
\eqref{eq:BTE} with the gain replaced by the elastic operator.

\subsection{Setting and error decomposition}
\label{subsec:errorcontrol-setting}

Let $\mathcal A$ denote either $\mathbb S^{d-1}$ or the cone
$\mathcal C(\vec\omega_\star,\theta_c)$ from \eqref{eq:cone}. Define the
$\mathcal A$-restricted elastic gain
\[
  \mathcal{K}_{\mathrm{el}}^{\mathcal A}[\psi](\vec{x},E,\vec{\omega})
  :=
  \int_{\mathcal A}
  \sigma_{\mathrm{el}}(\vec{\omega},\vec{\omega}',E) 
  \psi(\vec{x},E,\vec{\omega}') \d\vec{\omega}'.
\]
Let $\psi^{\mathcal A}$ denote the corresponding continuum solution of
\eqref{eq:BTE} with $\mathcal K_{\mathrm{el}}$ replaced by
$\mathcal K_{\mathrm{el}}^{\mathcal A}$. In particular,
$\psi^{\mathbb S^{d-1}}=\psi$.

Fix an angular domain $\mathcal A\subseteq \mathbb S^{d-1}$ (either
$\mathbb S^{d-1}$ or a cone) and an angular rule
$\{(\vec\omega_j,w_j)\}_{j=1}^Q\subset\mathcal A\times(0,\infty)$.
Let $\psi^{\mathcal A}$ denote the continuum solution of the elastic
transport problem in which the gain integral is restricted to
$\mathcal A$. Let $\{\psi_i^{Q}\}_{i=1}^Q$ denote the exact solution of
the coupled $Q$-direction discrete-ordinates system \eqref{eq:DOM} posed
on $\mathcal A$ with the same angular rule, and let
$\{\psi_i^{Q,n}\}_{i=1}^Q$ be the $n$th source-iteration iterate.

Then, for each node $\vec\omega_i$,
\begin{equation}
  \label{eq:error-split}
  \psi(\cdot,\cdot,\vec\omega_i)-\psi_i^{Q,n}
  =
  \big(\psi(\cdot,\cdot,\vec\omega_i)-\psi^{\mathcal A}(\cdot,\cdot,\vec\omega_i)\big)
  +
  \big(\psi^{\mathcal A}(\cdot,\cdot,\vec\omega_i)-\psi_i^{Q}\big)
  +
  \big(\psi_i^{Q}-\psi_i^{Q,n}\big).
\end{equation}
The first term is a model-truncation contribution, it vanishes when
$\mathcal A=\mathbb S^{d-1}$ (since then $\psi^{\mathcal A}=\psi$), and
when $\mathcal A$ is a cone it measures the effect of omitting
scattering into directions outside $\mathcal A$. The second term is the
discrete-ordinates (quadrature/collocation) contribution for the
$\mathcal A$-restricted model, it compares the continuum solution
$\psi^{\mathcal A}$ sampled at $\vec\omega_i$ with the exact coupled DOM
solution on the same ordinates. The third term is the source-iteration
contribution for solving the coupled DOM system.

To transfer gain consistency bounds into solution bounds we use the
directional coercivity norm induced by the transport form. For each
direction $\vec{\omega}_i$ define
\[
  \|w\|_{\mathcal U_i}^2
  :=
  \|\sigma_T^{1/2}w\|_{L^2(D\times I)}^2
  + \|(-S')^{1/2}w\|_{L^2(D\times I)}^2
  + \int_{\Gamma_+(\vec{\omega}_i)}
  \qp{\vec{\omega}_i\cdot\vec{n}_{\vec x}-S(E)n_E} w^2 \d\gamma,
\]
and the weighted product norm
\[
  \|\vec w\|_{\mathcal U^Q}^2 := \sum_{i=1}^Q \varpi_i \|w_i\|_{\mathcal U_i}^2.
\]

\begin{lemma}[Directional stability]
\label{lem:directional-stability}
Fix $\vec{\omega}_i$ and consider two right-hand sides
$F,\widetilde F\in \mathcal U_i'$ and two solutions
$u,\widetilde u\in\mathcal U_i$ with the same inflow data such that
\[
  a(u,v;\vec{\omega}_i)=F(v),
  \qquad
  a(\widetilde u,v;\vec{\omega}_i)=\widetilde F(v)
  \qquad \text{for all } v\in \mathcal U_i.
\]
Then
\[
  \|u-\widetilde u\|_{\mathcal U_i}
  \le
  2 \|F-\widetilde F\|_{\mathcal U_i'}.
\]
\end{lemma}

\begin{proof}
Subtract the two identities and test with $v=u-\widetilde u$. The error
has homogeneous inflow trace, so Theorem~\ref{the:coercivity} gives
$\tfrac12\|u-\widetilde u\|_{\mathcal U_i}^2
\le |(F-\widetilde F)(u-\widetilde u)|\le
\|F-\widetilde F\|_{\mathcal U_i'}\|u-\widetilde u\|_{\mathcal U_i}$.
\end{proof}

\subsection{Quadrature consistency on $\mathcal A$}
\label{subsec:errorcontrol-quadrature}

For an integrand $f:\mathcal A\to\mathbb R$ set
\[
  I_{\mathcal A}(f):=\int_{\mathcal A} f(\vec\omega) \d\vec\omega,
  \qquad
  Q_{\mathcal A}(f):=\sum_{j=1}^Q \mu_j f(\vec\omega_j).
\]
Define the angular fill distance
\[
  h_{\vec\omega}
  :=
  \sup_{\vec{\omega}\in\mathcal A}\min_{1\le j\le Q}
  \arccos\qp{\vec{\omega}\cdot\vec{\omega}_j}.
\]
If the rule is exact for spherical polynomials on $\mathcal A$ up to
total degree $m$ and $f\in W^{m+1,1}(\mathcal A)$, then there exists
$C=C(m,d)$ such that
\begin{equation}
  \label{eq:quad-error}
  |I_{\mathcal A}(f)-Q_{\mathcal A}(f)|
  \le
  C h_{\vec\omega}^{m+1}
  \sum_{|\alpha|=m+1}\|D_{\tan}^\alpha f\|_{L^1(\mathcal A)}.
\end{equation}

Applied pointwise in $(\vec x,E)$ to the elastic gain integrand
$f(\vec\omega')=\sigma_{\mathrm{el}}(\vec\omega_i,\vec\omega',E) 
\psi^{\mathcal A}(\vec x,E,\vec\omega')$, \eqref{eq:quad-error} bounds the
gain consistency defect
\[
  \qp{\mathcal K_{\mathrm{el}}^{\mathcal A}[\psi^{\mathcal A}]}
  (\vec x,E,\vec\omega_i)
  -
  \sum_{j=1}^Q \mu_j \sigma_{\mathrm{el}}(\vec\omega_i,\vec\omega_j,E) 
  \psi^{\mathcal A}(\vec x,E,\vec\omega_j).
\]
Using Lemma~\ref{lem:directional-stability} yields the corresponding bound
on $\psi_i^{\mathcal A}-\psi_i^Q$.

\begin{proposition}[Quadrature-induced angular discretisation error]
\label{prop:quad-to-solution}
Assume the hypotheses of Theorem~\ref{the:coercivity}. Assume moreover
that the angular rule on $\mathcal A$ is exact for spherical polynomials
up to degree $m$ and that for each $i$ the gain integrand
$\vec\omega'\mapsto \sigma_{\mathrm{el}}(\vec\omega_i,\vec\omega',E) 
\psi^{\mathcal A}(\vec x,E,\vec\omega')$ belongs to $W^{m+1,1}(\mathcal A)$
for a.e.\ $(\vec x,E)\in D\times I$, with tangential derivatives
integrable in $(\vec x,E)$. Then there exists $C_{\mathrm{ang}}=C(m,d)>0$
such that for each $i=1,\dots,Q$,
\[
  \|\psi_i^{\mathcal A}-\psi_i^{Q}\|_{\mathcal U_i}
  \le
  C_{\mathrm{ang}} h_{\vec\omega}^{m+1} 
  \Xi_{m+1}^{(i)}(\psi^{\mathcal A},\sigma_{\mathrm{el}}),
\]
where
\[
  \Xi_{m+1}^{(i)}(\psi^{\mathcal A},\sigma_{\mathrm{el}})
  :=
  \sum_{|\alpha|=m+1}
  \big\|D_{\tan}^\alpha\!\big(\sigma_{\mathrm{el}}(\vec\omega_i,\cdot,E) 
  \psi^{\mathcal A}(\vec x,E,\cdot)\big)\big\|_{L^1\big(D\times I;L^1(\mathcal A)\big)}.
\]
\end{proposition}

\subsection{Cone truncation}
\label{subsec:errorcontrol-truncation}

If $\mathcal A=\mathbb S^{d-1}$ this contribution vanishes. Otherwise
$\mathcal A=\mathcal C(\vec\omega_\star,\theta_c)$ and the truncated gain
omits scattering from $\mathbb S^{d-1}\setminus\mathcal C(\vec\omega_\star,\theta_c)$.
For any $f:\mathbb S^{d-1}\to\mathbb R$,
\[
  \int_{\mathbb S^{d-1}} f(\vec\omega') \d\vec\omega'
  -
  \int_{\mathcal C(\vec\omega_\star,\theta_c)} f(\vec\omega') \d\vec\omega'
  =
  \int_{\mathbb S^{d-1}\setminus\mathcal C(\vec\omega_\star,\theta_c)}
  f(\vec\omega') \d\vec\omega'.
\]
A crude estimate is
\begin{equation}
  \label{eq:cone_bound}
  \Big|\int_{\mathbb S^{d-1}\setminus\mathcal C(\vec\omega_\star,\theta_c)}
  f(\vec\omega') \d\vec\omega'\Big|
  \le
  \sup_{\vec\omega'\in\mathbb S^{d-1}\setminus\mathcal C(\vec\omega_\star,\theta_c)}
  |f(\vec\omega')| 
  \big|\mathbb S^{d-1}\setminus\mathcal C(\vec\omega_\star,\theta_c)\big|.
\end{equation}
Moreover, for forward-peaked kernels a more informative measure is the
cone tail
\[
  e_{\mathrm{cone}}(E)
  :=
  \sup_{\vec\omega\in\mathcal C(\vec\omega_\star,\theta_c)}
  \int_{\mathbb S^{d-1}\setminus\mathcal C(\vec\omega_\star,\theta_c)}
  \sigma_{\mathrm{el}}(\vec\omega,\vec\omega',E) \d\vec\omega'.
\]
This quantity measures the missing elastic gain at energy $E$. It enters
\eqref{eq:error-split} through the truncation term
$\psi(\cdot,\cdot,\vec\omega_i)-\psi^{\mathcal A}(\cdot,\cdot,\vec\omega_i)$.

\subsection{Source-iteration error and stopping}
\label{subsec:SIerror}

By Proposition~\ref{prop:SI-discrete}, the discrete source iteration is a
contraction in $\mathcal U^Q$ provided the discrete scattering operator
has bound $\eta_Q$ in $\|\cdot\|_{\mathcal U^Q}$ satisfying $2\eta_Q<1$.
With $\rho:=2\eta_Q\in(0,1)$ and $\vec u^{(\infty)}$ denoting the exact
solution of \eqref{eq:discrete_in_angle},
\begin{equation}
  \label{eq:iter-geometric}
  \|\vec u^{(n)}-\vec u^{(\infty)}\|_{\mathcal U^Q}
  \le
  \rho^n \|\vec u^{(0)}-\vec u^{(\infty)}\|_{\mathcal U^Q},
\end{equation}
and the contraction property yields the a posteriori estimate
\begin{equation}
  \label{eq:iter-apost}
  \|\vec u^{(n)}-\vec u^{(\infty)}\|_{\mathcal U^Q}
  \le
  \frac{\rho}{1-\rho} \|\vec u^{(n)}-\vec u^{(n-1)}\|_{\mathcal U^Q}.
\end{equation}
Hence, for a prescribed tolerance $\varepsilon>0$, it suffices to stop
when
\[
  \|\vec u^{(n)}-\vec u^{(n-1)}\|_{\mathcal U^Q}
  \le
  \frac{1-\rho}{\rho} \varepsilon.
\]

\begin{theorem}[Angular discretisation, truncation and iteration error]
\label{thm:combined-error}
Assume the hypotheses of Theorem~\ref{the:coercivity}. Fix $\mathcal
A\in\{\mathbb S^{d-1}, \mathcal C(\vec\omega_\star,\theta_c)\}$ and a
quadrature rule $\{(\vec\omega_j,\mu_j)\}_{j=1}^Q\subset\mathcal
A\times(0,\infty)$ with fill distance $h_{\vec\omega}$. Let $\psi$ be
the full-sphere continuum solution, let $\psi^{\mathcal A}$ be the
$\mathcal A$-restricted continuum solution and let $\psi_i^{Q,n}$ be
the $n$th source-iteration iterate for the coupled $Q$-direction
discrete-ordinates system posed on $\mathcal A$.

Assume that the quadrature hypotheses of Proposition~\ref{prop:quad-to-solution}
hold and that the discrete source iteration is contractive with factor
$\rho\in(0,1)$ so that \eqref{eq:iter-apost} holds. Then for each node
$\vec\omega_i$ there exist constants $C_{\mathrm{ang}}>0$ and
$C_{\mathrm{stab}}>0$ such that
\begin{equation}
  \label{eq:combined-error}
  \|\psi(\cdot,\cdot,\vec\omega_i)-\psi_i^{Q,n}\|_{\mathcal U_i}
  \le
  \|(\psi-\psi^{\mathcal A})(\cdot,\cdot,\vec\omega_i)\|_{\mathcal U_i}
  +
  C_{\mathrm{ang}} h_{\vec\omega}^{m+1} \Xi_{m+1}^{(i)}(\psi^{\mathcal A},\sigma_{\mathrm{el}})
  +
  C_{\mathrm{stab}}\frac{\rho}{1-\rho} 
  \|\vec u^{(n)}-\vec u^{(n-1)}\|_{\mathcal U^Q}.
\end{equation}
If $\mathcal A=\mathbb S^{d-1}$ the first term vanishes. If
$\mathcal A=\mathcal C(\vec\omega_\star,\theta_c)$, the first term is the
cone-truncation contribution and can be estimated in terms of the cone
tail $e_{\mathrm{cone}}(E)$ from Section~\ref{subsec:errorcontrol-truncation}
under additional control of $\psi$ outside the cone.
\end{theorem}

\begin{proof}
Apply the decomposition \eqref{eq:error-split}. Bound
$\psi_i^{\mathcal A}-\psi_i^Q$ by Proposition~\ref{prop:quad-to-solution}.
Bound $\psi_i^Q-\psi_i^{Q,n}$ by \eqref{eq:iter-apost} and the embedding
$\|\psi_i^Q-\psi_i^{Q,n}\|_{\mathcal U_i}\le
C_{\mathrm{stab}}\|\vec u^{(\infty)}-\vec u^{(n)}\|_{\mathcal U^Q}$.
\end{proof}

\begin{corollary}[Separable tolerance allocation]
\label{cor:errorcontrol-balance}
In the setting of Theorem~\ref{thm:combined-error}, a sufficient
condition for
$\|\psi(\cdot,\cdot,\vec\omega_i)-\psi_i^{Q,n}\|_{\mathcal U_i}\le \varepsilon$
is
\[
  \|(\psi-\psi_i^{\mathcal A})(\cdot,\cdot,\vec\omega_i)\|_{\mathcal U_i}
  \le \tfrac{\varepsilon}{3},
  \qquad
  C_{\mathrm{ang}} h_{\vec\omega}^{m+1} \Xi_{m+1}^{(i)}(\psi^{\mathcal A},\sigma_{\mathrm{el}})
  \le \tfrac{\varepsilon}{3},
  \qquad
  \frac{\rho}{1-\rho} \|\vec u^{(n)}-\vec u^{(n-1)}\|_{\mathcal U^Q}
  \le \tfrac{\varepsilon}{3C_{\mathrm{stab}}}.
\]
\end{corollary}

\begin{example}[Henyey--Greenstein on $\mathbb{S}^{1}$]
Let
\[
  p_{\gamma}(\theta)
  =
  \frac{1}{2\pi}\frac{1-\gamma^{2}}{1+\gamma^{2}-2\gamma\cos\theta},
\]
with $\theta$ the angle from $\vec{\omega}_\star$. For $\gamma=0.95$
and $\theta_c=\pi/2$, the supremum of $p_\gamma$ on the excluded set
$|\theta|\in[\theta_c,\pi]$ occurs at $|\theta|=\theta_c$, giving
\[
  \sup_{\theta\in[\theta_c,\pi]} p_\gamma(\theta)
  =
  \frac{1}{2\pi}\frac{1-\gamma^2}{1+\gamma^2-2\gamma\cos\theta_c}
  \approx 8.16\times 10^{-3}.
\]
Multiplying by the excluded arc length $2(\pi-\theta_c)=\pi$ yields the
bound $|I-I_{\mathcal{C}}|\lesssim 2.56\times 10^{-2}$. For
$\theta_c=3\pi/4$, the same calculation gives
$|I-I_{\mathcal{C}}|\lesssim 7.5\times 10^{-3}$.
\end{example}

\section{Numerical Experiments: Proton Transport}
\label{sec:numericsproton}

We now present numerical experiments to verify correctness of the
discretisation, assess convergence of the source iteration scheme and
investigate behaviour across representative forward-peaked scattering
regimes. All simulations are conducted using a uniform Cartesian mesh
and a piecewise constant angular discretisation unless stated otherwise.

\subsection{Implementation of the numerical algorithm}
\label{subsec:proton_implementation}

We solve the angularly discrete transport system by source iteration,
using the same discrete-ordinates quadrature framework as in
Section~\ref{sec:angular}. In particular, we work with quadrature nodes
$\{\vec{\omega}_j\}_{j=1}^Q$ and weights $\{\mu_j\}_{j=1}^Q$ on the
chosen angular domain (full sphere or cone). The angularly discrete
unknowns are the directional values
\[
  \Psi_j(\vec{x},E) := \psi(\vec{x},E,\vec{\omega}_j),
  \qquad j=1,\dots,Q.
\]

In these experiments we restrict to elastic Coulomb scattering and
therefore use the elastic kernel $\sigma_{\mathrm{el}}$ from
Section~\ref{sec:angular}. The discrete elastic gain at node
$\vec{\omega}_i$ is approximated by the midpoint quadrature rule,
\begin{equation}
  \label{eq:rhs_scattering_numeric}
  \int_{\mathcal A}
  \sigma_{\mathrm{el}}(\vec{\omega}_i,\vec{\omega}',E) 
  \psi(\vec{x},E,\vec{\omega}') \d\vec{\omega}'
   \approx 
  \sum_{j=1}^Q \mu_j 
  \bar\sigma_{\mathrm{el}}(\vec{\omega}_i,\vec{\omega}_j,E) 
  \Psi_j(\vec{x},E),
\end{equation}
where $\bar\sigma_{\mathrm{el}}(\vec{\omega}_i,\vec{\omega}_j,E)$ is a
quadrature-consistent approximation of the kernel evaluated at the node
$\vec{\omega}_j$ (in the simplest case,
$\bar\sigma_{\mathrm{el}}(\vec{\omega}_i,\vec{\omega}_j,E)=
\sigma_{\mathrm{el}}(\vec{\omega}_i,\vec{\omega}_j,E)$).

To ensure discrete conservation of total fluence (gain--loss balance), we
enforce a discrete analogue of the identity in
Proposition~\ref{prop:conservation}. Writing
\[
  \Sigma_{\mathrm{el}}(E)
  :=
  \int_{\mathbb S^{d-1}}
  \sigma_{\mathrm{el}}(\vec{\omega},\vec{\omega}',E) \d\vec{\omega},
\]
(which is independent of $\vec{\omega}'$ under rotational invariance),
we impose the weighted row-sum condition
\begin{equation}
  \label{eq:discrete_fluence_conservation}
  \sum_{i=1}^Q \varpi_i \bar\sigma_{\mathrm{el}}(\vec{\omega}_i,\vec{\omega}_j,E)
  =
  \Sigma_{\mathrm{el}}(E),
  \qquad j=1,\dots,Q.
\end{equation}
Equivalently, for any angular vector $\{\Psi_j(\vec{x},E)\}_{j=1}^Q$ one has
\[
  \sum_{i=1}^Q \varpi_i
  \sum_{j=1}^Q \mu_j \bar\sigma_{\mathrm{el}}(\vec{\omega}_i,\vec{\omega}_j,E) 
  \Psi_j(\vec{x},E)
  =
  \Sigma_{\mathrm{el}}(E)\sum_{j=1}^Q \mu_j \Psi_j(\vec{x},E),
\]
so the discrete elastic gain balances the corresponding loss in the
weighted angular sum.

\begin{remark}[Relation to a finite-volume interpretation in angle]
Although we implement scattering via quadrature nodes and weights, the
formula \eqref{eq:rhs_scattering_numeric} can be read as a finite-volume
exchange between angular control volumes: the weights $\mu_j$ play the
role of angular ``cell measures'', while
$\bar\sigma_{\mathrm{el}}(\vec{\omega}_i,\vec{\omega}_j,E)$ acts as an
energy-dependent transfer rate from direction $\vec{\omega}_j$ to
$\vec{\omega}_i$. Condition \eqref{eq:discrete_fluence_conservation}
then corresponds to conservation of total fluence in the angular
finite-volume balance.

When the angular domain is restricted to a forward cone, some fluence
is necessarily lost at the boundary of the cone. In regimes where the
scattering kernel is strongly forward-peaked (as is typical in proton
and ion transport), this loss is negligible in practice.
\end{remark}

The angularly discrete system seeks functions $\Psi_i:D\times I\to\mathbb{R}$
for $i=1,\dots,Q$ satisfying
\begin{equation}
  \label{eq:BTEdiscrete}
  \vec{\omega}_i \cdot \nabla_{\vec{x}} \Psi_i(\vec{x},E)
  -
  \partial_E\qp{S(E)\Psi_i(\vec{x},E)}
  +
  \sigma_T(E)\Psi_i(\vec{x},E)
  =
  \sum_{j=1}^Q \mu_j \bar\sigma_{\mathrm{el}}(\vec{\omega}_i,\vec{\omega}_j,E) \Psi_j(\vec{x},E),
\end{equation}
which is the discrete-ordinates model \eqref{eq:DOM} with the elastic
scattering sum realised via the quadrature-consistent coefficients
$\bar\sigma_{\mathrm{el}}$. We solve \eqref{eq:BTEdiscrete} by the source
iteration method introduced in Section~\ref{subsec:sourceiteration}. The
algorithmic structure is summarised in Algorithm~\ref{alg:iterative_algorithm}.

\begin{algorithm}[t]
\caption{Source iteration with method-of-characteristics (MoC) directional sweeps}
\label{alg:iterative_algorithm}
\begin{algorithmic}[1]
\Require Quadrature nodes $\{\vec{\omega}_i\}_{i=1}^Q$ and weights $\{\varpi_i\}_{i=1}^Q$, Stopping power $S(E)$, removal $\sigma_T(E)$, elastic coupling $\bar\sigma_{\mathrm{el}}(\vec{\omega}_i,\vec{\omega}_j,E)$, Inflow data $g(\vec{x},E,\vec{\omega}_i)$ on $\Gamma_-(\vec{\omega}_i)$ for $i=1,\dots,Q$, Tolerance $\mathrm{TOL}>0$, maximum iterations $N_{\max}\in\mathbb N$
\Ensure An iterate $\vec{u}^{(n)}=(u_1^{(n)},\dots,u_Q^{(n)})$ with
$\|\vec{u}^{(n)}-\vec{u}^{(n-1)}\|_{\mathcal{U}^Q}\le \mathrm{TOL}$, or $n=N_{\max}$

\Statex

\State \textbf{Initialisation (ballistic MoC sweep):}
\For{$i=1$ to $Q$}
  \State Set $G_i^{(0)}(\vec{x},E)\gets 0$
  \State Compute $u_i^{(0)}$ by a MoC sweep in direction $\vec{\omega}_i$:
  integrate along characteristics of the drift $(\vec{\omega}_i,-S(E))$
  with inflow trace $g(\cdot,\cdot,\vec{\omega}_i)$ and source $G_i^{(0)}$
\EndFor

\Statex

\For{$n=1$ to $N_{\max}$}
  \State \textbf{Elastic gain assembly:}
  \For{$i=1$ to $Q$}
    \State $G_i^{(n)}(\vec{x},E)\gets \sum_{j=1}^Q \mu_j \bar\sigma_{\mathrm{el}}(\vec{\omega}_i,\vec{\omega}_j,E) u_j^{(n-1)}(\vec{x},E)$
  \EndFor

  \State \textbf{Directional MoC sweeps:}
  \For{$i=1$ to $Q$}
    \State Compute $u_i^{(n)}$ by a MoC sweep in direction $\vec{\omega}_i$:
    evaluate $u_i^{(n)}(\vec{x},E)$ via the characteristic map and the
    line integral of $G_i^{(n)}$ along characteristics of $(\vec{\omega}_i,-S(E))$,
    with inflow trace $g(\cdot,\cdot,\vec{\omega}_i)$
  \EndFor

  \State $\mathrm{ERR}\gets \|\vec{u}^{(n)}-\vec{u}^{(n-1)}\|_{\mathcal{U}^Q}$
  \If{$\mathrm{ERR}\le \mathrm{TOL}$}
    \State \textbf{break}
  \EndIf
\EndFor
\end{algorithmic}
\end{algorithm}

\subsection{Dose functional and discrete dose notation}
\label{subsec:dose_def}

The transport solver produces an approximation of the phase-space
density $\psi(\vec x,E,\vec\omega)$. In the numerical experiments we
primarily report dose-like outputs obtained by integrating $\psi$ against
an energy-deposition weight. Concretely, for a given stopping power
$S(E)$ we define the (unnormalised) dose functional
\begin{equation}
  \label{eq:dose_def}
  \mathscr D(\vec x)
  :=
  \int_{I}\int_{\mathcal A} S(E) \psi(\vec x,E,\vec\omega) \d\vec\omega \d E,
\end{equation}
where $\mathcal A=\mathbb S^{d-1}$ in the full angular model and
$\mathcal A=\mathcal C(\vec\omega_\star,\theta_c)$ when a cone is used.
(Any constant factors converting energy deposition to physical dose in
Gy are absorbed into the overall normalisation in these tests.)

For the angularly discrete model with quadrature nodes
$\{(\vec\omega_i,\varpi_i)\}_{i=1}^Q\subset\mathcal A\times(0,\infty)$, let
$\Psi_i(\vec x,E)$ denote the computed directional solution at
$\vec\omega_i$. The corresponding discrete dose is defined by replacing
the angular integral in \eqref{eq:dose_def} by quadrature and evaluating
the remaining energy integral numerically on the chosen energy grid:
\begin{equation}
  \label{eq:dose_discrete}
  \mathscr D^{Q}(\vec x)
  :=
  \sum_{k=1}^{N_E} w_k \sum_{i=1}^Q \varpi_i  S(E_k) \Psi_i(\vec x,E_k),
\end{equation}
where $\{E_k,w_k\}_{k=1}^{N_E}$ are the energy nodes and weights used in
the implementation. When a reference solution is available analytically,
we compute the corresponding reference dose $\mathscr D$ from \eqref{eq:dose_def}
using the same energy quadrature to isolate the spatial discretisation
error.

\subsection{Experiment 1: Benchmarking against an exact transport solution}
\label{exp1:benchmark}

We first consider a degenerate limit in which scattering is absent in
practice because the angular redistribution collapses to the identity. We
model this by the distributional kernel
\begin{equation}
  \label{eq:dirac_scattering}
  \sigma_{\mathrm{el}}(\vec{\omega},\vec{\omega}',E)
  =
  \Sigma_{\mathrm{el}}(E) \delta(\vec{\omega}-\vec{\omega}'),
\end{equation}
so that
$\mathcal{K}_{\mathrm{el}}[\psi](\vec{x},E,\vec{\omega})
=\Sigma_{\mathrm{el}}(E) \psi(\vec{x},E,\vec{\omega})$. Choosing the
removal to match, $\sigma_T(E)=\Sigma_{\mathrm{el}}(E)$, the gain term
cancels the loss term and \eqref{eq:BTE} reduces to the pure
transport--slowing-down equation
\[
  \vec{\omega}\cdot\nabla_{\vec{x}}\psi(\vec{x},E,\vec{\omega})
  -
  \partial_E\qp{S(E)\psi(\vec{x},E,\vec{\omega})}
  =
  0,
\]
posed with the inflow condition $\psi=g$ on $\Gamma_-$. This case admits
an explicit characteristic solution and therefore provides a stringent
code-level benchmark for the method-of-characteristics sweep, including
correct localisation along $\vec{\omega}$ and correct propagation in
energy.

In the Bragg--Kleeman setting $S(E)=\frac{1}{\alpha p}E^{1-p}$, an exact
solution is
\begin{equation}
  \label{eq:exact_dirac_solution}
  \psi(\vec{x},E,\vec{\omega})
  =
  \qp{
    E^p + \frac{\vec{\omega}\cdot(\vec{x}-\vec{x}_0)}{\alpha}
  }^{\frac{1-p}{p}}
  E^{p-1} 
  g\qp{
    \qp{
      E^p + \frac{\vec{\omega}\cdot(\vec{x}-\vec{x}_0)}{\alpha}
    }^{\frac{1}{p}},
    \vec{x}_0,\vec{\omega}
  },
\end{equation}
where $\vec{x}_0=\vec{x}_0(\vec{x},E,\vec{\omega})\in\partial D$ denotes
the point where the characteristic through $(\vec{x},E)$ in direction
$\vec{\omega}$ meets the spatial inflow boundary (equivalently, the
preimage under the characteristic map).

We compare the numerical approximation produced by the MoC sweep
against the exact benchmark \eqref{eq:exact_dirac_solution} through
the induced dose field
\eqref{eq:dose_def}. Figure~\ref{fig:benchmark_error_distribution}
displays the pointwise relative dose error
$\qp{\mathscr D^{Q}-\mathscr D}/\Norm{\mathscr D}_{\leb{\infty}}$, and
Figure~\ref{fig:benchmark_slice_plot} shows the corresponding computed
dose field $\mathscr D^{Q}$ for this ballistic transport case.

\begin{figure}[h!]
  \centering
  \begin{subfigure}{.45\linewidth}
    \includegraphics[width=\textwidth]{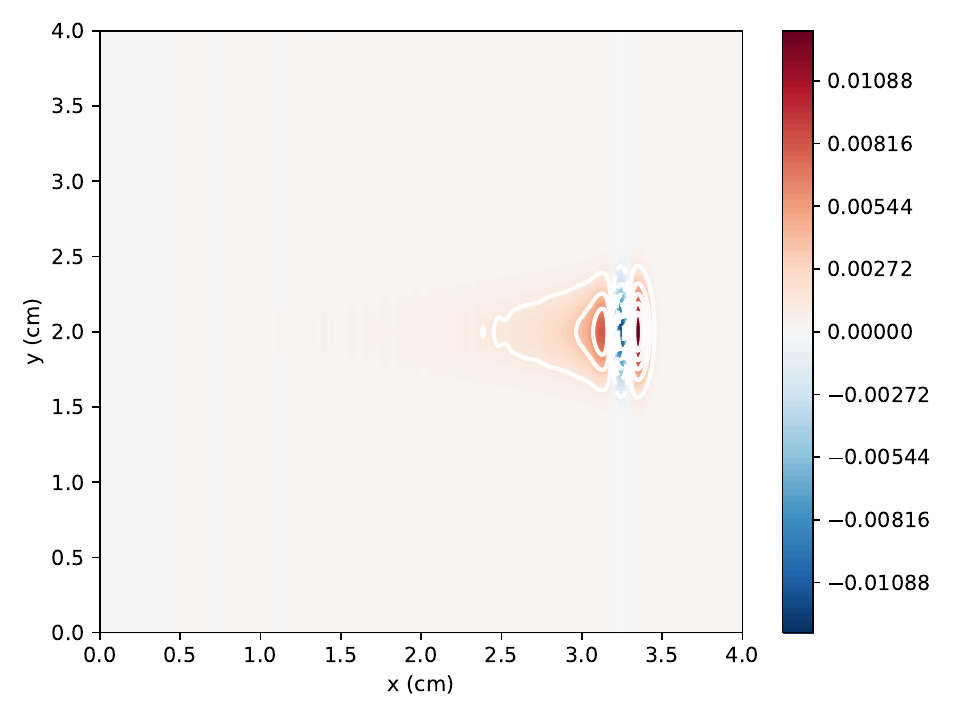}
    \caption{Pointwise relative error for the ballistic benchmark.}
    \label{fig:benchmark_error_distribution}
  \end{subfigure}
  \hfil
  \begin{subfigure}{.45\linewidth}
    \includegraphics[width=\textwidth]{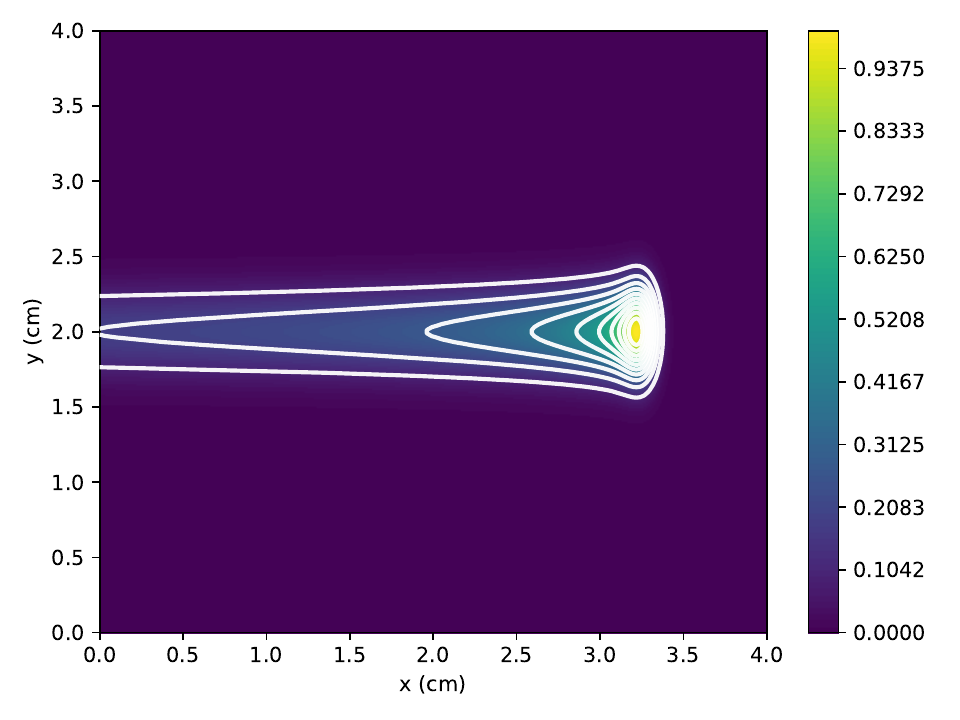}
    \caption{Computed dose field corresponding to the ballistic benchmark solution.}
    \label{fig:benchmark_slice_plot}
  \end{subfigure}
  \caption{Ballistic (degenerate scattering) benchmark. Left: pointwise relative
    error of the numerical transport solution against the exact characteristic
    solution \eqref{eq:exact_dirac_solution}. Right: computed dose field for the
    same test case.}
\end{figure}

To quantify convergence under mesh refinement, we report the maximum
pointwise dose error normalised by the peak reference dose,
\[
  \frac{\|\mathscr D^{Q}-\mathscr D\|_{L^\infty(D)}}{\|\mathscr D\|_{L^\infty(D)}}
   =
  \frac{\max_{\vec{x}\in D}|\mathscr D^{Q}(\vec{x})-\mathscr D(\vec{x})|}{\max_{\vec{x}\in D}\mathscr D(\vec{x})}.
\]
and plot it against the number of spatial cells $N_x$ on a log--log
scale. Figure~\ref{fig:benchmark_convergence_rate} also shows an
empirical power-law fit obtained by linear regression in logarithmic
coordinates.

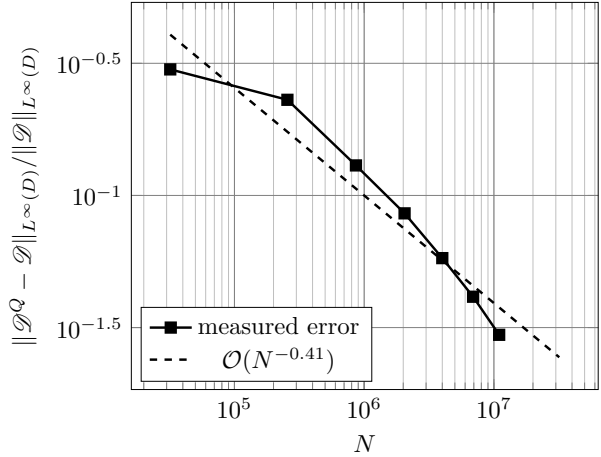
\begin{figure}[h!]
  \centering
  \begin{subfigure}{.5\linewidth}
  \begin{tikzpicture}[scale=0.93]
  \pgfplotstableread[col sep=comma]{experiment1output/benchmark_grid_study.csv}\datatable

  \pgfplotstablecreatecol[
    create col/expr={\thisrow{Nx}*\thisrow{Ny}*\thisrow{Ne}}
  ]{N}{\datatable}

  \pgfplotstablecreatecol[
    create col/expr={ln(\thisrow{N})}
  ]{lnN}{\datatable}
  \pgfplotstablecreatecol[
    create col/expr={ln(\thisrow{max_rel_to_peak})}
  ]{lny}{\datatable}

  \pgfplotstablecreatecol[
    linear regression={x=lnN,y=lny}
  ]{linreg}{\datatable}

  \pgfmathsetmacro{\rate}{-\pgfplotstableregressiona}
  \pgfmathsetmacro{\pref}{exp(\pgfplotstableregressionb)}

  \begin{axis}[
      width=\linewidth,
      xmode=log,
      ymode=log,
      grid=both,
      major grid style={black!50},
      xlabel={$N$},
      ylabel={$\|\mathscr D^Q-\mathscr D\|_{L^\infty(D)}/\|\mathscr D\|_{L^\infty(D)}$},
      legend style={at={(0.02,0.02)},anchor=south west},
  ]

    \addplot[
      solid,
      mark=square*,
      mark options={scale=1, solid},
      color={black!100},
      line width=1.0
    ]
    table[x=N, y=max_rel_to_peak]{\datatable};
    \addlegendentry{measured error}

    \addplot[
      dashed,
      color={black!100},
      line width=1.0,
      domain=32000:32000000,
      samples=200
    ]{\pref * x^(-\rate)};
    \addlegendentry{$\mathcal{O}(N^{-\pgfmathprintnumber[fixed,precision=2]{\rate}})$}

  \end{axis}
  \end{tikzpicture}
  \end{subfigure}
  \caption{\label{fig:benchmark_convergence_rate} Convergence of the
    ballistic benchmark under spatial mesh refinement, measured by the
    maximum pointwise dose error normalised by the peak reference
    dose, $\|\mathscr D^Q-\mathscr D\|_{L^\infty(D)}/\|\mathscr
    D\|_{L^\infty(D)}$, plotted against $N=N_xN_yN_E$ on a log--log
    scale. The dashed line shows the empirical power-law fit.}
\end{figure}

\subsection{Experiment 2: Convergence of source iteration}
\label{exp2:SI}

We next examine convergence of the source-iteration scheme for the
angularly discrete transport system on a fixed discretisation. We use
the finest spatial and energy grids from
Experiment~\ref{exp1:benchmark} and the same angular rule as in the
subsequent scattering runs. Fix an angular rule with $Q$
directions. Let
$\vec{\psi}^{ (n)}=(\psi^{(n)}_1,\dots,\psi^{(n)}_Q)\in U^Q$ denote
the $n$th iterate produced by source iteration, where each component
$\psi^{(n)}_i(\vec{x},E)\approx\psi(\vec{x},E,\vec{\omega}_i)$ is
computed by a single method-of-characteristics sweep in direction
$\vec{\omega}_i$ with a right-hand side assembled from the previous
iterate.

To quantify convergence, we monitor the maximum successive-iterate
difference over all directions and phase-space gridpoints,
\begin{equation}
  \label{eq:SI_diff_inf}
  \Delta_\infty^{(n)}
  :=
  \max_{1\le i\le Q} 
  \|\psi^{(n)}_i-\psi^{(n-1)}_i\|_{L^\infty(D\times I)}.
\end{equation}
This diagnostic is inexpensive to compute, is directly available from
the stored iterates and provides a practical stopping proxy. In
particular, when the discrete source iteration is contractive in the
product graph norm, the a~posteriori bound \eqref{eq:iter-apost}
implies that small successive differences
$\|\vec{\psi}^{ (n)}-\vec{\psi}^{ (n-1)}\|_{\mathcal U^Q}$ control
the distance to the converged discrete-ordinates solution. While
$\Delta_\infty^{(n)}$ is not the same norm, in practice it tracks the
same geometric decay and is therefore used here as a convenient
stopping indicator.

Figure~\ref{fig:algo_convergence} plots $\Delta_\infty^{(n)}$ against
the iteration count $n$ on a logarithmic scale. For comparison we also
plot a geometric reference curve of the form $C r^{ n-1}$ with
$C=\Delta_\infty^{(1)}$ and a representative contraction factor
$r\in(0,1)$.  When the source iteration is contractive, one expects
$\Delta_\infty^{(n)}$ to decay approximately geometrically until the
iteration reaches the level of discretisation and floating-point
errors.

\begin{figure}[h!]
  \centering
  \begin{subfigure}{.5\linewidth}
  \begin{tikzpicture}[scale=0.93]

    \pgfplotstableread[col sep=comma]{experiment1output/benchmark_finest_iter_history.csv}\datatable

    \pgfmathsetmacro{\rref}{0.4}

    \pgfplotstablegetelem{0}{diff_inf}\of{\datatable}
    \pgfmathsetmacro{\Cref}{\pgfplotsretval}

    \begin{axis}[
      width=\linewidth,
      ymode=log,
      grid=both,
      major grid style={black!50},
      xlabel = {Source iterations $n$},
      ylabel = {$\Delta_\infty^{(n)}$},
      legend style={at={(0.02,0.02)},anchor=south west},
    ]

      \addplot[
        solid,
        mark=square*,
        mark options={scale=1, solid},
        color={black!100},
        line width=1.0
      ]
      table[x=iter, y=diff_inf]{\datatable};
      \addlegendentry{$\Delta_\infty^{(n)}$}

      \pgfmathprintnumberto[fixed,precision=2]{\rref}{\rreflegend}
      \addplot[
        dashed,
        color={black!100},
        line width=1.0,
        domain=1:20,
        samples=200
      ]{\Cref * (\rref)^(x-1)};
      \addlegendentry{$C (\rreflegend)^{ n-1}$ (reference)}

    \end{axis}
  \end{tikzpicture}
  \end{subfigure}
  \caption{\label{fig:algo_convergence} Convergence history of source
    iteration on the finest discretisation, measured by the
    successive-iterate difference $\Delta_\infty^{(n)}=\max_{1\le i\le
      Q}\|\psi^{Q,n}_i-\psi^{Q,n-1}_i\|_{L^\infty(D\times I)}$. The
    dashed curve is a geometric reference line $C r^{ n-1}$ with
    $C=\Delta_\infty^{(1)}$ and a representative $r\in(0,1)$.}
\end{figure}
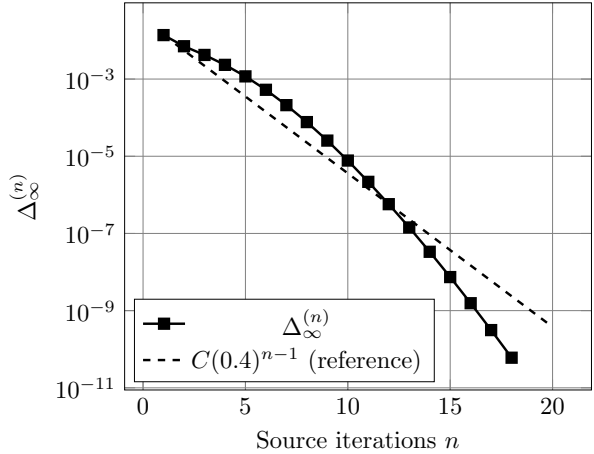

\subsection{Experiment 3: Simulations with a Henyey--Greenstein cross section}
\label{exp3:HG}

To model anisotropic (forward-peaked) angular scattering we use the
Henyey--Greenstein phase function on $\mathbb S^{1}$,
\begin{equation}
  \label{eq:HG_phase}
  p_{\gamma}(\theta)
  =
  \frac{1}{2\pi}\frac{1-\gamma^2}{1+\gamma^2-2\gamma\cos\theta},
  \qquad \theta\in[-\pi,\pi),
\end{equation}
which satisfies the normalisation $\int_{-\pi}^{\pi} p_{\gamma}(\theta) \d\theta = 1$.
For $\gamma\in(0,1)$ the distribution is strongly forward-peaked, whereas
$\gamma<0$ corresponds to backscattering.

In these experiments we work in $d=2$, so directions are parametrised by
an angle $\theta$ via $\vec\omega(\theta)=(\cos\theta,\sin\theta)$.
We take an energy-independent, rotationally invariant elastic kernel of
the form
\[
  \sigma_{\mathrm{el}}\!\big(\vec\omega(\theta),\vec\omega(\theta')\big)
  =
  \Sigma_{\mathrm{el}} p_{\gamma}(\theta-\theta'),
\]
so that the corresponding gain is the circular convolution
\[
  \mathcal K_{\mathrm{el}}[\psi](\vec x,E,\theta)
  =
  \Sigma_{\mathrm{el}}\int_{-\pi}^{\pi} p_{\gamma}(\theta-\theta') 
  \psi(\vec x,E,\theta') \d\theta'.
\]
This kernel is isotropic about the current direction in the sense that it
depends only on the relative turning angle $\theta-\theta'$.

Figure~\ref{fig:hg_phase} shows $p_{\gamma}$ for $\gamma=0.9$ together
with its piecewise-constant approximation on the angular bins used in
the computation. Concretely, if $\{B_j\}_{j=1}^Q$ is a partition of
$[-\pi,\pi)$ with bin measures $\mu_j:=|B_j|$ and midpoints
$\theta_j\in B_j$, then the bar height in bin $B_j$ is the bin average
$\bar p_j := \mu_j^{-1}\int_{B_j} p_{\gamma}(\theta) \d\theta$, so that
$p_{\gamma}\approx \sum_{j=1}^Q \bar p_j \mathbf 1_{B_j}$ and
$\sum_{j=1}^Q \mu_j \bar p_j = 1$.

\begin{figure}[h!]
  \centering
  \includegraphics[width=0.55\linewidth]{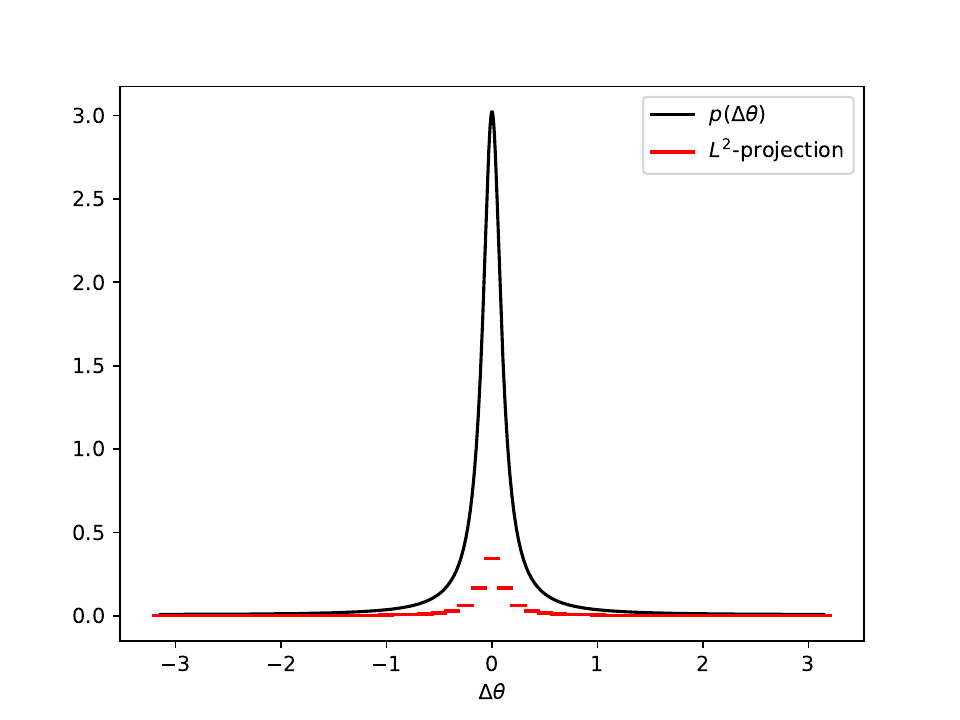}
  \caption{\label{fig:hg_phase}
  Henyey--Greenstein phase function \eqref{eq:HG_phase} on $\mathbb S^{1}$
  with anisotropy parameter $\gamma=0.9$. The red bars show the bin
  averages $\bar p_j$ on the piecewise-constant angular partition used to
  assemble the discrete scattering operator.}
\end{figure}

After convergence of source iteration (Experiment~\ref{exp2:SI}), we
obtain the discrete-ordinates solution $\vec\psi^{ Q}$ and compute
the associated dose field $\mathscr D^{Q}$ (as defined in the dose
subsection preceding Experiment~\ref{exp1:benchmark}). To highlight
the effect of angular scattering, Figure~\ref{fig:hg_gamma_sweep}
places a number of choices of Henyey--Greenstein parameters. The dose
exhibits increased lateral spread for smaller $\gamma$, reflecting
redistribution of fluence away from purely characteristic transport
directions.

\begin{figure}[h!]
  \centering

  \begin{subfigure}{.48\linewidth}
    \centering
    \includegraphics[width=\linewidth]{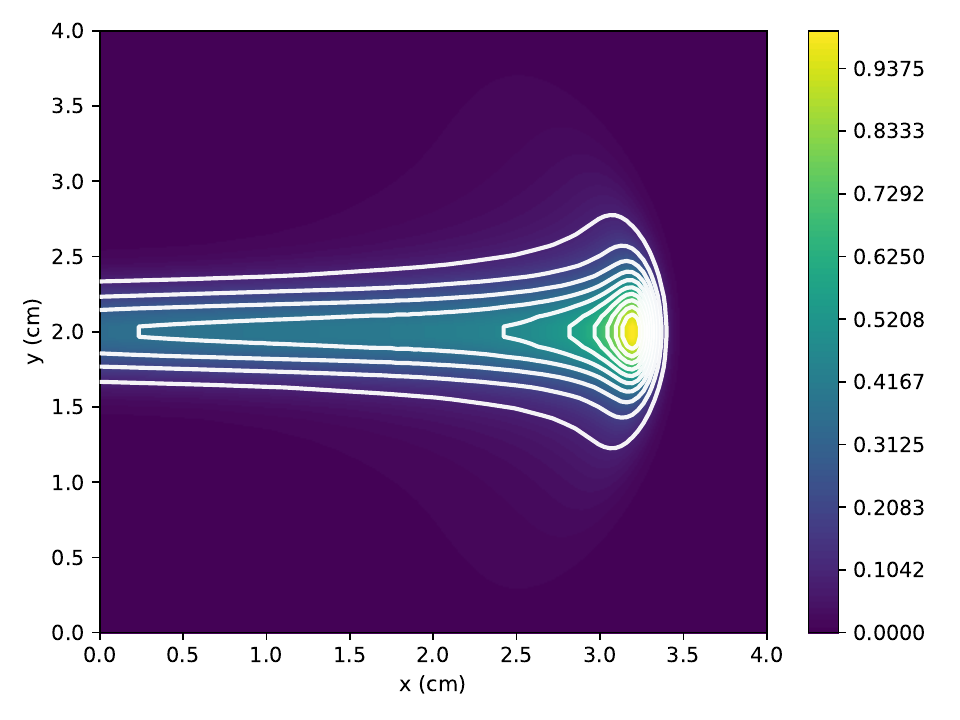}
    \caption{Henyey--Greenstein scattering, $\gamma=0.95$.}
    \label{fig:hg_dose_095}
  \end{subfigure}
  \hfil
  \begin{subfigure}{.48\linewidth}
    \centering
    \includegraphics[width=\linewidth]{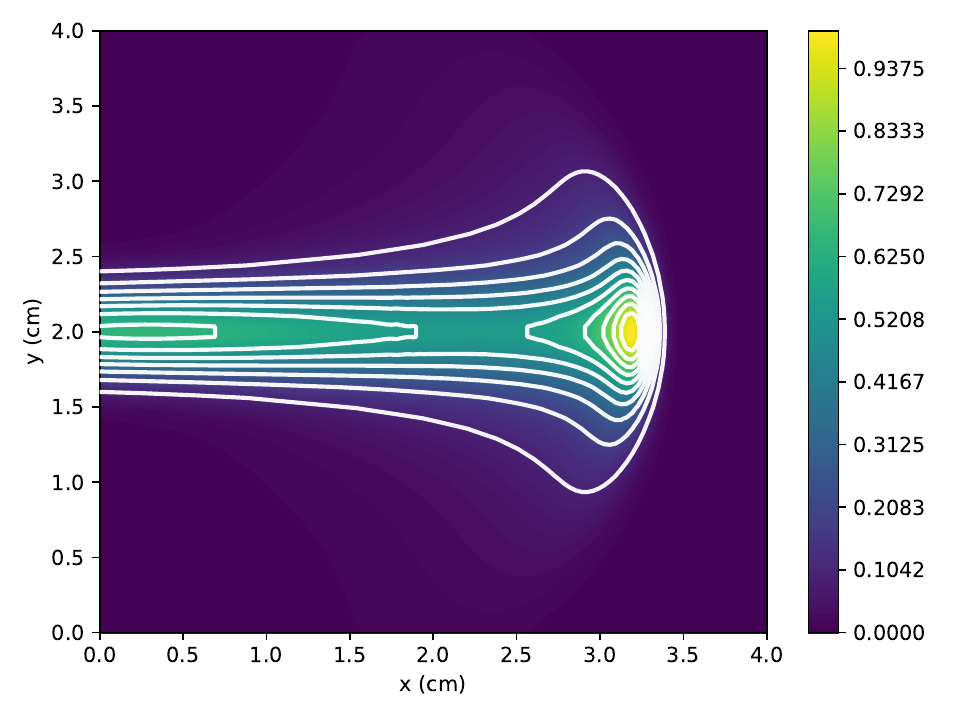}
    \caption{Henyey--Greenstein scattering, $\gamma=0.90$.}
    \label{fig:hg_dose_090}
  \end{subfigure}

  \vspace{0.6em}

  \begin{subfigure}{.48\linewidth}
    \centering
    \includegraphics[width=\linewidth]{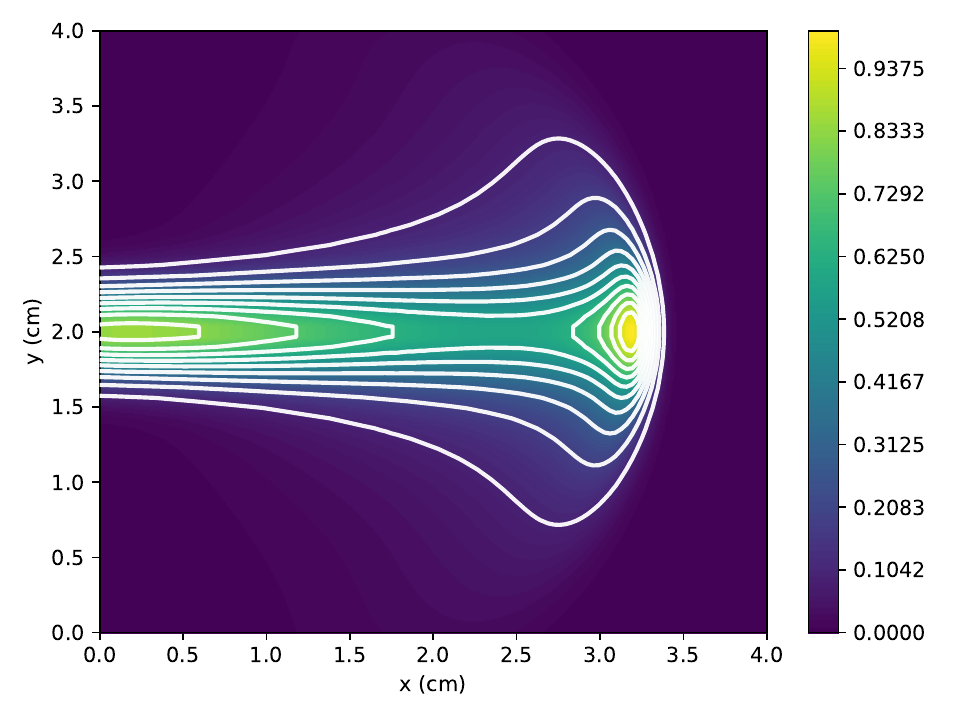}
    \caption{Henyey--Greenstein scattering, $\gamma=0.85$.}
    \label{fig:hg_dose_085}
  \end{subfigure}
  \hfil
  \begin{subfigure}{.48\linewidth}
    \centering
    \includegraphics[width=\linewidth]{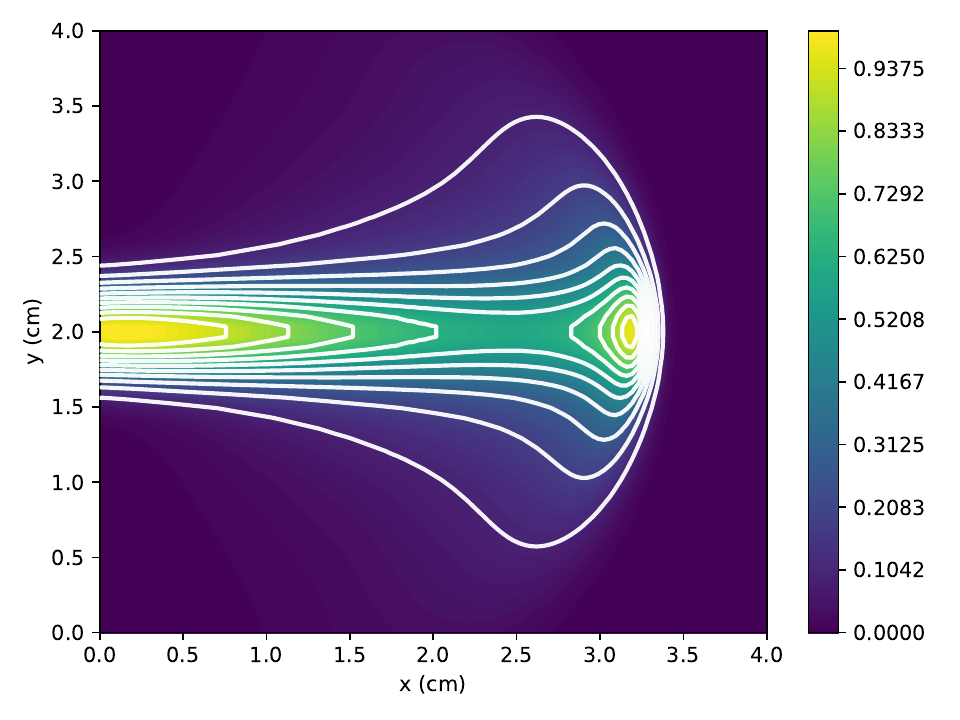}
    \caption{Henyey--Greenstein scattering, $\gamma=0.80$.}
    \label{fig:hg_dose_080}
  \end{subfigure}

  \caption{\label{fig:hg_gamma_sweep}
  Dose fields $D^{Q}$ computed with Henyey--Greenstein scattering
  \eqref{eq:HG_phase} for $\gamma\in\{0.95,0.90,0.85,0.80\}$. Decreasing
  $\gamma$ induces visible lateral broadening of the beam relative to
  the more forward-peaked cases.}
\end{figure}

\subsection{Experiment 4: Angular discretisation study (ordinates and cone truncation)}
\label{exp4:angular_study}

This experiment quantifies the effect of the angular approximation on
the computed proton dose field. In the notation of
Section~\ref{sec:angular}, the angular approximation has two distinct
components: the number of discrete ordinates $Q$ (angular resolution)
and, when the angular domain is restricted, the cone half-angle
$\theta_c$ (truncation). We therefore study both effects within a single
test configuration, with all other discretisation parameters held fixed.

We fix the spatial mesh, the energy grid and the stopping power $S(E)$
as in the preceding experiments. We model elastic scattering with the
Henyey--Greenstein phase function \eqref{eq:HG_phase} at a
representative forward-peaked parameter $\gamma=0.9$. For each angular
discretisation, we run source iteration to a sufficiently tight
tolerance so that the residual effect of iteration is negligible
compared to angular discretisation effects (this is validated a
posteriori using the successive-iterate diagnostic from
Experiment~\ref{exp2:SI}).

To isolate the influence of $Q$ and $\theta_c$ we compare against a
numerically converged angular reference. Concretely, we compute a
full-sphere discrete-ordinates reference solution
$\vec\psi^{ Q_{\mathrm{ref}}}$ with a large number of directions
$Q_{\mathrm{ref}}$, and define the corresponding reference dose
$D^{Q_{\mathrm{ref}}}$. For cone-restricted runs, we also consider a
cone-reference at the same $\theta_c$ but with a large $Q$, denoted
$D^{Q_{\mathrm{ref}},\theta_c}$, to distinguish cone truncation
effects from pure quadrature resolution effects.

For any run with $Q$ directions and cone half-angle $\theta_c$ we
compute the dose field $D^{Q,\theta_c}$ (as defined in the dose
subsection preceding Experiment~\ref{exp1:benchmark}) and report the
relative dose error against the chosen reference
\[
  \mathcal{E}_\infty(Q,\theta_c)
  :=
  \frac{\|D^{Q,\theta_c}-D^{\mathrm{ref}}\|_{L^\infty(D)}}{\|D^{\mathrm{ref}}\|_{L^\infty(D)}},
  \qquad
  D^{\mathrm{ref}}\in\{D^{Q_{\mathrm{ref}}}, D^{Q_{\mathrm{ref}},\theta_c}\}.
\]
We additionally report a simple lateral-spread diagnostic at a fixed
depth $x=x_\dagger$ (chosen in the mid-range of the track), defined as a
second-moment beam width
\[
  W^{Q,\theta_c}(x_\dagger)
  :=
  \Bigg(
  \frac{\int_{y} (y-\bar y)^2 D^{Q,\theta_c}(x_\dagger,y) \d y}
       {\int_{y} D^{Q,\theta_c}(x_\dagger,y) \d y}
  \Bigg)^{1/2},
  \qquad
  \bar y
  :=
  \frac{\int_y y D^{Q,\theta_c}(x_\dagger,y) \d y}{\int_y D^{Q,\theta_c}(x_\dagger,y) \d y},
\]
which is stable under mesh refinement and directly reflects the
scattering-induced lateral broadening.

To separate cone truncation effects from angular-resolution effects we
use two distinct angular references.  First, we compute a
high-resolution maximal-cone reference dose
$D^{Q_{\mathrm{ref}},\theta_{\max}}$, where $\theta_{\max}$ denotes
the maximal cone used in the implementation (in the present 2D setting
$\theta_{\max}=\pi/2$, that is, all forward directions). This
reference serves as a proxy for the untruncated forward-direction
model.

Second, for each cone half-angle $\theta_c$ in the sweep we compute a
high-resolution cone reference $D^{Q_{\mathrm{ref}},\theta_c}$ on the
same truncated angular domain. Comparing against
$D^{Q_{\mathrm{ref}},\theta_c}$ therefore isolates the
angular-resolution error inside the cone, since both solutions neglect
out-of-cone scattering in exactly the same way.

Two parameter sweeps are performed.

\begin{enumerate}
\item
  Ordinates study. We fix the cone half-angle at its maximal value used in the
  implementation, denoted $\theta_{\max}$. In the present 2D setting this
  corresponds to $\theta_{\max}=\pi/2$, that is, all forward directions.
  We then vary the number of discrete ordinates $Q$ and compare the resulting
  dose fields $D^{Q,\theta_{\max}}$ against the maximal-cone reference
  $D^{Q_{\mathrm{ref}},\theta_{\max}}$ (here $Q_{\mathrm{ref}}=257$). We report the
  relative peak-normalised dose error $\mathcal{E}_\infty(Q,\theta_{\max})$ and
  the beam-width diagnostic $W^{Q,\theta_{\max}}(x_\dagger)$ defined above.

\item
  Cone-truncation study. We fix the angular resolution at $Q=Q_\star$ (here
  $Q_\star=33$) and vary the cone half-angle
  $\theta_c\in\{\theta_{c,1},\dots,\theta_{\max}\}$. For each $\theta_c$ we report
  two errors: $\mathcal{E}_\infty(Q_\star,\theta_c)$ measured against the maximal-cone
  reference $D^{Q_{\mathrm{ref}},\theta_{\max}}$ (capturing the combined effect of
  cone truncation and angular resolution) and the corresponding error measured
  against the cone reference $D^{Q_{\mathrm{ref}},\theta_c}$ (isolating the
  angular-resolution effect inside the truncated cone), with $Q_{\mathrm{ref}}=129$
  for this sweep. Note that, with $Q_\star$ fixed, enlarging the cone simultaneously
  decreases truncation error and coarsens the angular partition, so the error against
  $D^{Q_{\mathrm{ref}},\theta_{\max}}$ reflects a trade-off between truncation and
  resolution and need not be monotone in $\theta_c$.
\end{enumerate}

Figure~\ref{fig:angular_error_summary} collects the two error curves, one for
ordinates refinement and one for cone truncation. Figure~\ref{fig:angular_width_study}
shows the stabilisation of the beam-width diagnostic under ordinates refinement.
Finally, Figure~\ref{fig:angular_cone_snapshots} provides representative dose
snapshots for four cone half-angles, illustrating the qualitative impact of
truncation.

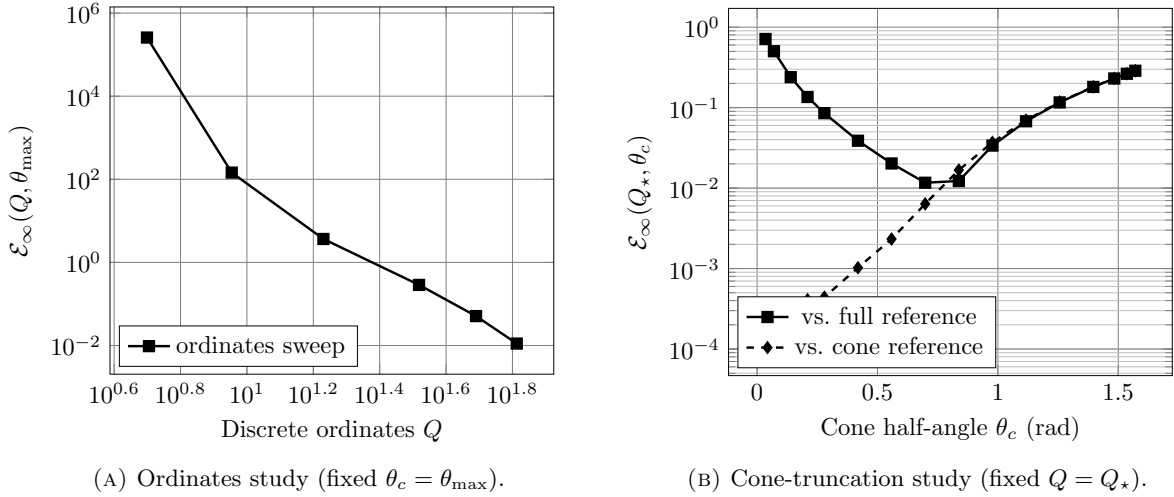
\begin{figure}[h!]
  \centering
  \begin{subfigure}{.48\linewidth}
  \begin{tikzpicture}[scale=0.93]

    \pgfplotstableread[col sep=comma]{experiment4output/angular_Q_study.csv}\datatableQ

    \begin{axis}[
      width=\linewidth,
      xmode=log,
      ymode=log,
      grid=both,
      major grid style={black!50},
      xlabel = {Discrete ordinates $Q$},
      ylabel = {$\mathcal{E}_\infty(Q,\theta_{\max})$},
      legend style={at={(0.02,0.02)},anchor=south west},
    ]

      \addplot[
        solid,
        mark=square*,
        mark options={scale=1, solid},
        color={black!100},
        line width=1.0
      ]
      table[x=Q, y=E_inf]{\datatableQ};
      \addlegendentry{ordinates sweep}

    \end{axis}
  \end{tikzpicture}
  \caption{Ordinates study (fixed $\theta_c=\theta_{\max}$).}
  \end{subfigure}
  \hfil
  \begin{subfigure}{.48\linewidth}
  \begin{tikzpicture}[scale=0.93]

    \pgfplotstableread[col sep=comma]{experiment4output/angular_cone_study.csv}\datatableC

    \begin{axis}[
      width=\linewidth,
      ymode=log,
      grid=both,
      major grid style={black!50},
      xlabel = {Cone half-angle $\theta_c$ (rad)},
      ylabel = {$\mathcal{E}_\infty(Q_\star,\theta_c)$},
      legend style={at={(0.02,0.02)},anchor=south west},
    ]

      \addplot[
        solid,
        mark=square*,
        mark options={scale=1, solid},
        color={black!100},
        line width=1.0
      ]
      table[x=theta_c, y=E_inf_fullref]{\datatableC};
      \addlegendentry{vs.\ full reference}

      \addplot[
        dashed,
        mark=diamond*,
        mark options={scale=1, solid},
        color={black!100},
        line width=1.0
      ]
      table[x=theta_c, y=E_inf_coneref]{\datatableC};
      \addlegendentry{vs.\ cone reference}

    \end{axis}
  \end{tikzpicture}
  \caption{Cone-truncation study (fixed $Q=Q_\star$).}
  \end{subfigure}
  \caption{\label{fig:angular_error_summary} Angular discretisation
    error study with Henyey--Greenstein scattering at $\gamma=0.9$.
    Left: convergence of the dose error under ordinates refinement at
    fixed $\theta_c=\theta_{\max}$.  Right: error versus cone
    half-angle at fixed $Q=Q_\star$. The curve shown against the
    maximal-cone reference captures both cone truncation and
    angular-resolution effects, while the curve shown against the cone
    reference isolates the angular-resolution effect inside the
    truncated cone.}
\end{figure}

\begin{figure}[h!]
  \centering
  \begin{subfigure}{.62\linewidth}
  \begin{tikzpicture}[scale=0.93]

    \pgfplotstableread[col sep=comma]{experiment4output/angular_Q_study.csv}\datatableQ

    \pgfplotstablegetrowsof{\datatableQ}
    \pgfmathtruncatemacro{\lastrow}{\pgfplotsretval-1}
    \pgfplotstablegetelem{\lastrow}{W}\of{\datatableQ}
    \pgfmathsetmacro{\Wref}{\pgfplotsretval}

    \begin{axis}[
      width=\linewidth,
      xmode=log,
      grid=both,
      major grid style={black!50},
      xlabel = {Discrete ordinates $Q$},
      ylabel = {$W^{Q,\theta_{\max}}(x_\dagger)$},
      legend style={at={(0.02,0.02)},anchor=south west},
    ]

      \addplot[
        solid,
        mark=square*,
        mark options={scale=1, solid},
        color={black!100},
        line width=1.0
      ]
      table[x=Q, y=W]{\datatableQ};
      \addlegendentry{measured width}

      \addplot[
        dashed,
        color={black!100},
        line width=1.0,
        domain=5:300,
        samples=2
      ]{\Wref};
      \addlegendentry{reference level}

    \end{axis}
  \end{tikzpicture}
  \end{subfigure}
  \caption{\label{fig:angular_width_study} Beam-width diagnostic under
    ordinates refinement (Henyey--Greenstein, $\gamma=0.9$), shown at
    the fixed depth $x=x_\dagger$. The dashed line indicates the value
    obtained at the finest ordinates level in the sweep.}
\end{figure}
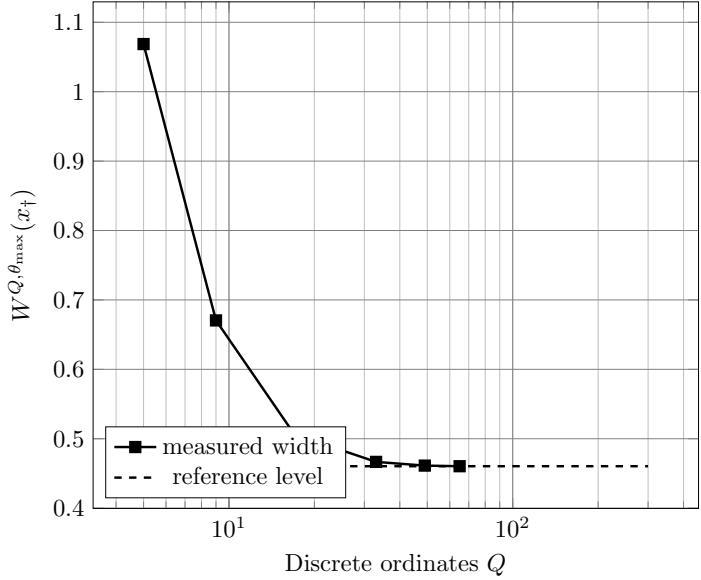

\begin{figure}[h!]
  \centering
  \begin{subfigure}{.48\linewidth}
    \includegraphics[width=\textwidth]{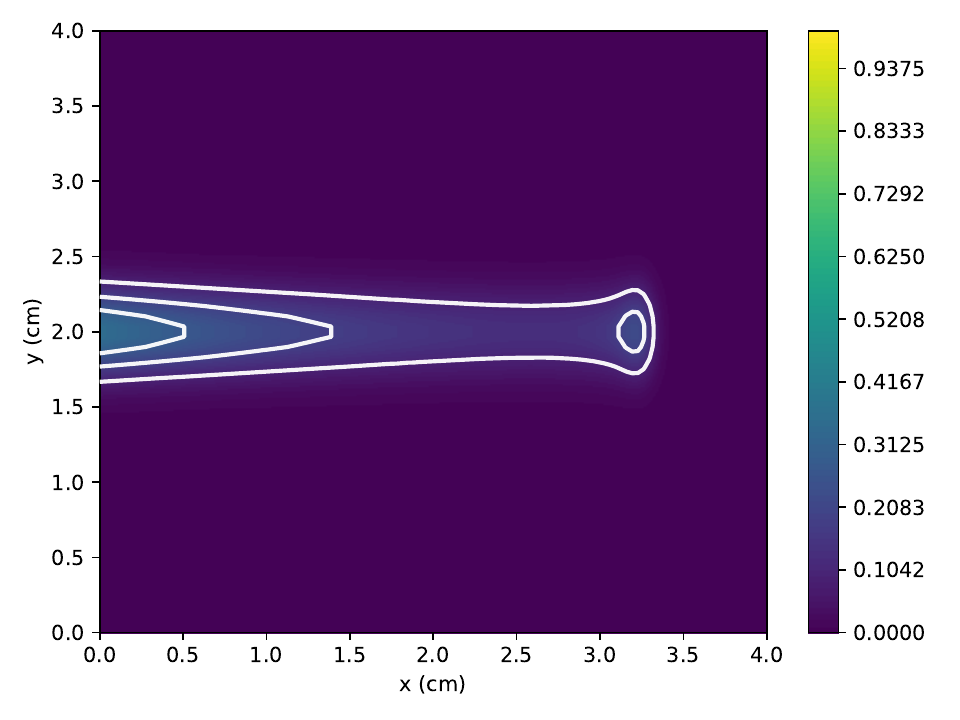}
    \caption{$\theta_c=0.0349$ rad ($2^\circ$).}
  \end{subfigure}
  \hfil
  \begin{subfigure}{.48\linewidth}
    \includegraphics[width=\textwidth]{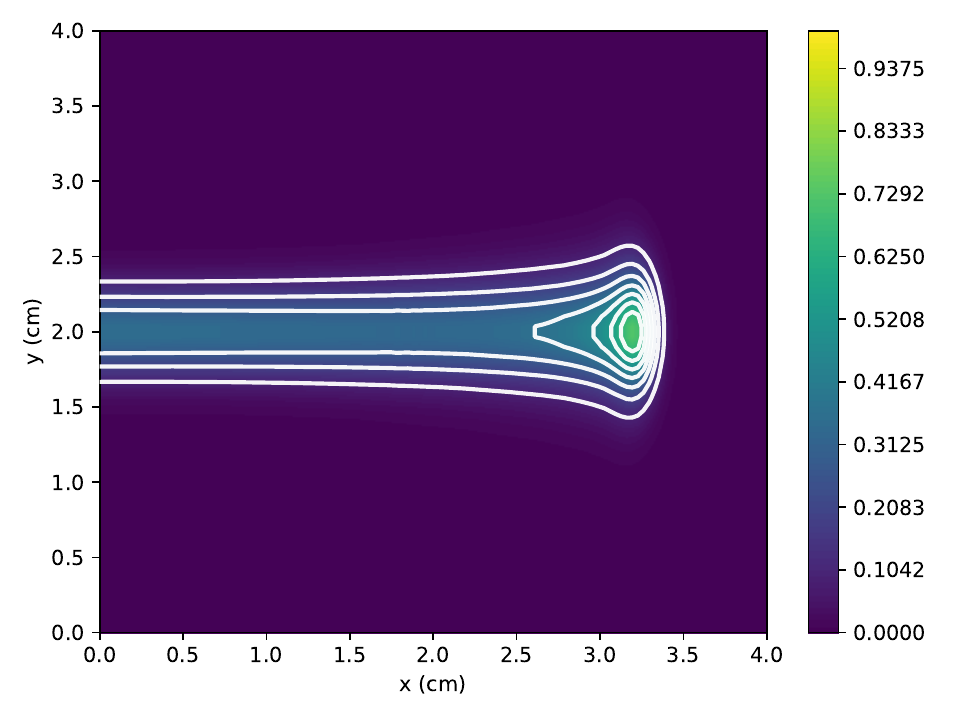}
    \caption{$\theta_c=0.2793$ rad ($16^\circ$).}
  \end{subfigure}

  \vspace{1ex}

  \begin{subfigure}{.48\linewidth}
    \includegraphics[width=\textwidth]{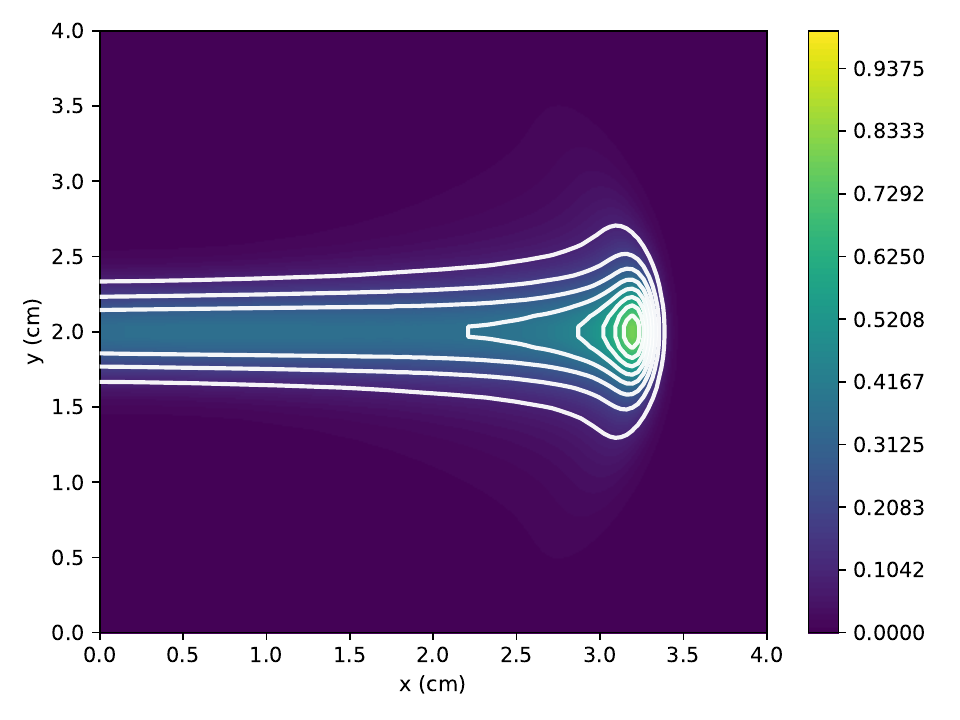}
    \caption{$\theta_c=0.8378$ rad ($48^\circ$).}
  \end{subfigure}
  \hfil
  \begin{subfigure}{.48\linewidth}
    \includegraphics[width=\textwidth]{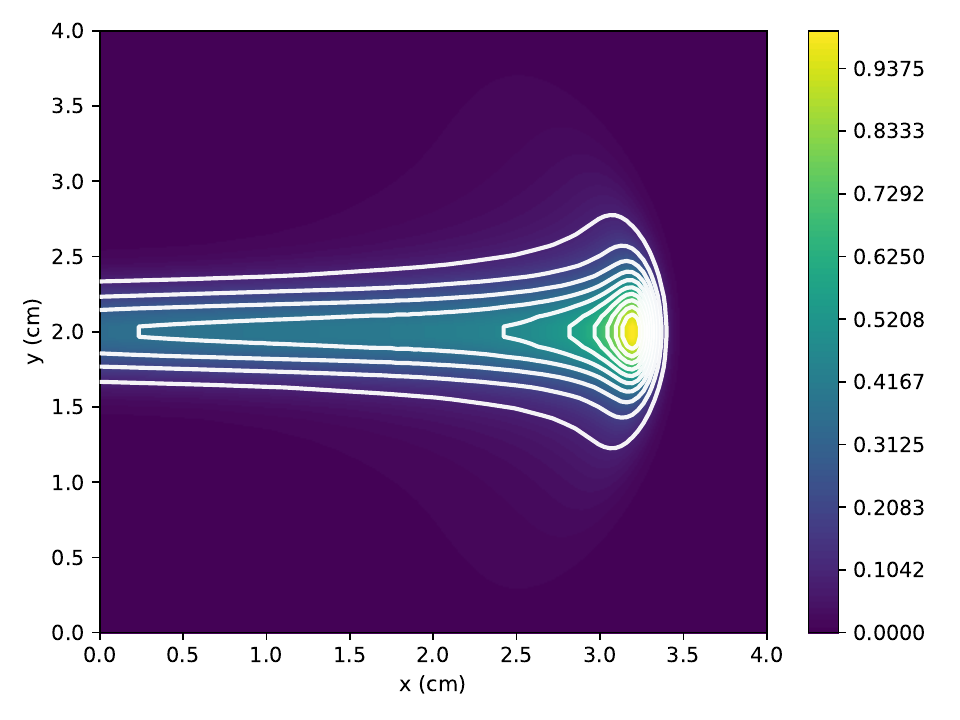}
    \caption{$\theta_c=1.5708$ rad ($90^\circ$).}
  \end{subfigure}
  \caption{\label{fig:angular_cone_snapshots} Representative dose
    fields $D^{Q_\star,\theta_c}$ from the cone-truncation sweep
    (Henyey--Greenstein, $\gamma=0.9$, fixed $Q=Q_\star$),
    illustrating how overly small cones under-represent wide-angle
    scattering and suppress lateral spread.}
\end{figure}

\subsection{Experiment 5: Coupling through scattering and iteration behaviour}
\label{exp5:coupling_iteration}

This experiment assesses the practical coupling strength induced by
elastic scattering in the discrete-ordinates system and its impact on
the convergence of source iteration. The analysis in
Section~\ref{subsec:SIerror} predicts linear convergence under a
subcriticality condition, with slower contraction as the effective
scattering strength increases. Here we verify this behaviour in a
physically meaningful regime.

We fix the spatial mesh, energy grid, stopping power
$S(E)$ and angular discretisation (either full sphere at a chosen $Q$ or
a cone with prescribed $\theta_c$) as in
Experiment~\ref{exp4:angular_study}. We consider a family of
Henyey--Greenstein parameters $\gamma\in(0,1)$, spanning moderately
anisotropic to strongly forward-peaked scattering. For each $\gamma$ we
run the source iteration scheme to a fixed tolerance $\mathrm{TOL}$ and
record both the successive-iterate diagnostic and the iteration count.

Let $\vec\psi^{ Q,n}=(\psi^{Q,n}_1,\dots,\psi^{Q,n}_Q)$
denote the $n$th iterate of the angularly discrete system, where each
$\psi^{Q,n}_i(\vec{x},E)$ is obtained by one method-of-characteristics
sweep in direction $\vec\omega_i$ with a right-hand side assembled from
$\vec\psi^{ Q,n-1}$. To quantify convergence we monitor the maximum
successive-iterate difference
\begin{equation}
  \label{eq:SI_diff_inf_exp5}
  \Delta_\infty^{(n)}
  :=
  \max_{1\le i\le Q} 
  \|\psi_i^{Q,n}-\psi_i^{Q,n-1}\|_{L^\infty(D\times I)}.
\end{equation}
This quantity is inexpensive to compute, is directly available from the
stored iterates and serves as a practical stopping proxy for the
a~posteriori bound in Section~\ref{subsec:SIerror}. We also report the
iteration count $n_{\mathrm{it}}(\gamma)$ required to achieve
$\Delta_\infty^{(n)}\le\mathrm{TOL}$.

We present two figures. Figure~\ref{fig:SI_gamma_histories} compares
the decay of $\Delta_\infty^{(n)}$ for several values of $\gamma$,
demonstrating that the iteration contracts more slowly as scattering
becomes more strongly forward-peaked and hence more redistributive
within the discretised angular
domain. Figure~\ref{fig:SI_gamma_counts} reports the iteration count
$n_{\mathrm{it}}(\gamma)$ as a function of $\gamma$.

\begin{figure}[h!]
  \centering
  \begin{subfigure}{.48\linewidth}
  \begin{tikzpicture}[scale=0.93]
    \pgfplotstableread[col sep=comma]{experiment5output/SI_gamma_counts.csv}\datatable
    \begin{axis}[
      width=\linewidth,
      grid=both,
      major grid style={black!50},
      xlabel={Henyey--Greenstein parameter $\gamma$},
      ylabel={Iterations to tolerance},
      legend style={at={(0.02,0.02)},anchor=south west},
    ]
      \addplot[
        solid,
        mark=square*,
        mark options={scale=1, solid},
        color={black!100},
        line width=1.0
      ]
      table[x=gamma, y=n_iter]{\datatable};
      \addlegendentry{$n_{\mathrm{iter}}$}
    \end{axis}
  \end{tikzpicture}
  \caption{Source-iteration counts vs.\ anisotropy parameter $\gamma$.}
  \end{subfigure}
  \hfil
  \begin{subfigure}{.48\linewidth}
  \begin{tikzpicture}[scale=0.93]
    \pgfplotstableread[col sep=comma]{experiment5output/SI_gamma_counts.csv}\datatable
    \begin{axis}[
      width=\linewidth,
      ymode=log,
      grid=both,
      major grid style={black!50},
      xlabel={Henyey--Greenstein parameter $\gamma$},
      ylabel={Final $\Delta_\infty^{(n)}$},
      legend style={at={(0.02,0.02)},anchor=south west},
    ]
      \addplot[
        solid,
        mark=square*,
        mark options={scale=1, solid},
        color={black!100},
        line width=1.0
      ]
      table[x=gamma, y=last_Delta_inf]{\datatable};
    \end{axis}
  \end{tikzpicture}
  \caption{Final successive-iterate diagnostic at termination.}
  \end{subfigure}
  \caption{\label{fig:SI_gamma_counts} Dependence of source-iteration
    performance on the scattering anisotropy parameter $\gamma$.}
\end{figure}
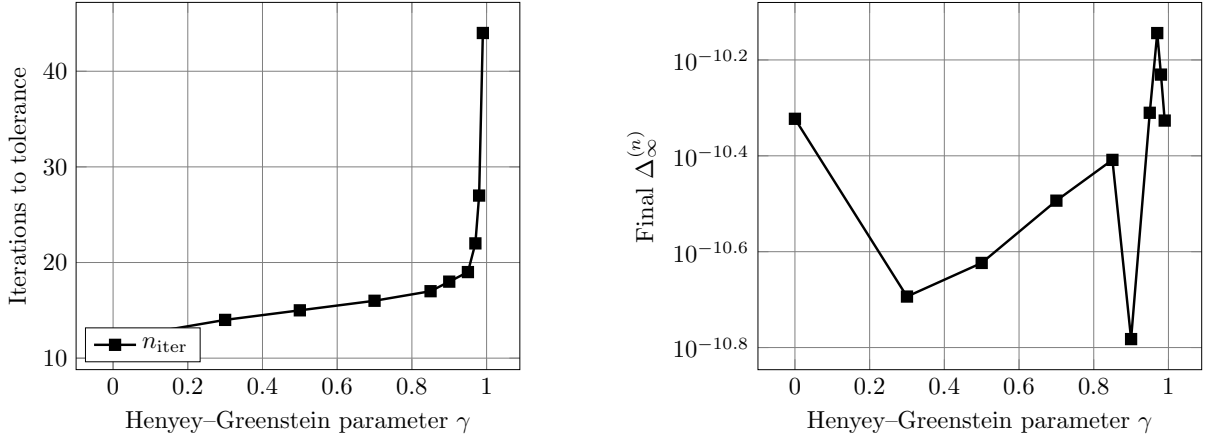

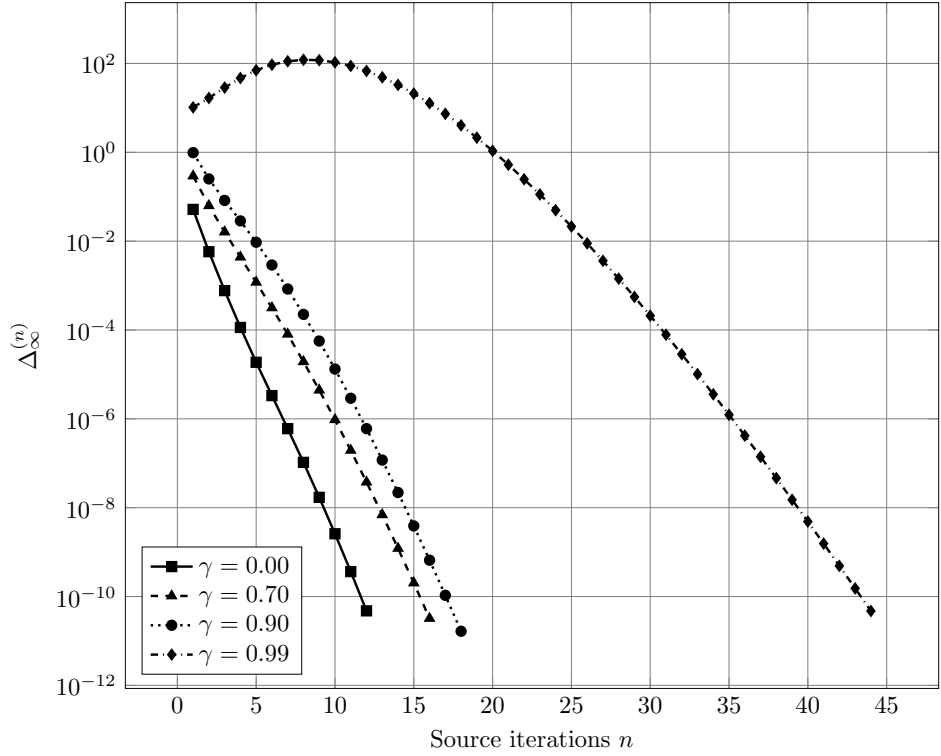
\begin{figure}[h!]
  \centering
  \begin{subfigure}{.8\linewidth}
  \begin{tikzpicture}[scale=0.93]
    \begin{axis}[
      width=\linewidth,
      ymode=log,
      grid=both,
      major grid style={black!50},
      xlabel={Source iterations $n$},
      ylabel={$\Delta_\infty^{(n)}$},
      legend style={at={(0.02,0.02)},anchor=south west},
    ]
      \pgfplotstableread[col sep=comma]{experiment5output/SI_gamma_history_g0p0000.csv}\dataA
      \addplot[
        solid,
        mark=square*,
        mark options={scale=0.9, solid},
        color={black!100},
        line width=1.0
      ] table[x=iter, y=Delta_inf]{\dataA};
      \addlegendentry{$\gamma=0.00$}

      \pgfplotstableread[col sep=comma]{experiment5output/SI_gamma_history_g0p7000.csv}\dataB
      \addplot[
        dashed,
        mark=triangle*,
        mark options={scale=0.9, solid},
        color={black!100},
        line width=1.0
      ] table[x=iter, y=Delta_inf]{\dataB};
      \addlegendentry{$\gamma=0.70$}

      \pgfplotstableread[col sep=comma]{experiment5output/SI_gamma_history_g0p9000.csv}\dataC
      \addplot[
        dotted,
        mark=*,
        mark options={scale=0.9, solid},
        color={black!100},
        line width=1.0
      ] table[x=iter, y=Delta_inf]{\dataC};
      \addlegendentry{$\gamma=0.90$}

      \pgfplotstableread[col sep=comma]{experiment5output/SI_gamma_history_g0p9900.csv}\dataD
      \addplot[
        dashdotted,
        mark=diamond*,
        mark options={scale=0.9, solid},
        color={black!100},
        line width=1.0
      ] table[x=iter, y=Delta_inf]{\dataD};
      \addlegendentry{$\gamma=0.99$}

    \end{axis}
  \end{tikzpicture}
  \end{subfigure}
  \caption{\label{fig:SI_gamma_histories} Convergence histories
    $\Delta_\infty^{(n)}=\max_{1\le i\le
      Q}\|\psi_i^{Q,n}-\psi_i^{Q,n-1}\|_{L^\infty(D\times I)}$ for
    representative Henyey--Greenstein parameters $\gamma$.}
\end{figure}

\section{Numerical Experiments: Carbon Ion Transport}
\label{sec:numericscarbon}

Unlike protons, carbon ions undergo nuclear interactions and
fragmentation in water, producing a broad spectrum of secondary
particles. The goal of this section is a demonstration of a
multi-species workflow built on the same characteristic sweeping
infrastructure as the proton solver. A scattered carbon primary field
is computed first, then secondary protons and neutrons are generated
from the converged carbon solution through prescribed volumetric
sources and transported in reduced models. Our emphasis is
out-of-field exposure, so we focus on neutrons (long-range transport
and relevance to protection quantities) and include secondary protons
as a simple charged-fragment proxy.

\subsection{Experiment 6: Out-of-field secondaries in carbon transport (protons and neutrons)}
\label{expC:multispecies}

Carbon primaries are transported with continuous slowing down driven by
a tabulated stopping power. Let $S_C(E)$ denote the linear stopping
power (MeV/cm) for carbon ions in water expressed in terms of the total
ion energy $E$ (MeV). As in the proton experiments, we work with an
angularly discrete representation and compute directional carbon fluence
$\Psi_i^C:D\times I\to\mathbb{R}$, $i=1,\dots,Q$, from
\begin{equation}
  \label{eq:carbon_DOM}
  \vec{\omega}_i \cdot \nabla_{\vec{x}} \Psi^C_i(\vec{x},E)
  -
  \partial_E\qp{S_C(E) \Psi^C_i(\vec{x},E)}
  +
  \sigma_T^C(E) \Psi^C_i(\vec{x},E)
  =
  \sum_{j=1}^Q \mu_j \bar\sigma_{\mathrm{el}}^C(\vec{\omega}_i,\vec{\omega}_j) \Psi^C_j(\vec{x},E),
\end{equation}
supplemented by a forward inflow carbon beam (supported only in the
central angular node, with the remaining inflow bins set to zero). The
right-hand side is treated by source iteration, with each iteration step
consisting of $Q$ independent method-of-characteristics sweeps. The
characteristic map required by the sweep is defined using a fitted
Bragg--Kleeman range law $R_{\mathrm{BK}}(E)=\alpha_C E^{p_C}$, while all
dose postprocessing uses the tabulated $S_C(E)$.

From the converged carbon solution we form the scalar carbon fluence and
dose
\begin{equation}
  \label{eq:carbon_primary_dose}
  \Phi^C(\vec{x},E) := \sum_{i=1}^Q \varpi_i \Psi^C_i(\vec{x},E),
  \qquad
  D_C^{Q}(\vec{x}) := \int_I S_C(E) \Phi^C(\vec{x},E) \d E.
\end{equation}
To drive secondaries we introduce an effective macroscopic nuclear
interaction rate $\Sigma_{\mathrm{nuc}}(E)$ and define the interaction
proxy
\begin{equation}
  \label{eq:interaction_proxy}
  \mathcal{I}^C(\vec{x})
  :=
  \int_I \Sigma_{\mathrm{nuc}}(E) \Phi^C(\vec{x},E) \d E,
\end{equation}
which concentrates secondary production in regions of high primary
fluence and high interaction probability. Secondary protons and neutrons
are then generated by volumetric sources that factorise into a spatial
production rate $\mathcal I^C(\vec x)$, an emitted-energy density and a
simple angular distribution. Concretely, we take
\begin{equation}
  \label{eq:secondary_sources}
  Q_i^P(\vec{x},E_P)
  =
  Y_P \mathcal{I}^C(\vec{x}) w_P(E_P) \frac{1}{Q},
  \qquad
  Q_i^N(\vec{x},E_N)
  =
  Y_N \mathcal{I}^C(\vec{x}) w_N(E_N) \frac{1}{Q},
\end{equation}
where $Y_P$ and $Y_N$ are yield parameters, while $w_P$ and $w_N$ are
prescribed energy densities. The uniform $1/Q$ factor corresponds to
an isotropic redistribution over the discrete ordinates; this can be
replaced by a forward-peaked emission model or by a data-driven
angular distribution without changing the solver structure.

Given the sources, secondary protons are transported with ionisation
loss only,
\begin{equation}
  \label{eq:secondary_proton_transport}
  \vec{\omega}_i \cdot \nabla_{\vec{x}} \Psi^P_i(\vec{x},E_P)
  -
  \partial_{E_P}\qp{S_P(E_P) \Psi^P_i(\vec{x},E_P)}
  =
  Q_i^P(\vec{x},E_P),
\end{equation}
with no inflow contribution, while neutrons are propagated in a reduced
streaming--removal model with no energy loss term,
\begin{equation}
  \label{eq:secondary_neutron_transport}
  \vec{\omega}_i \cdot \nabla_{\vec{x}} \Psi^N_i(\vec{x},E_N)
  +
  \sigma_T^N(E_N) \Psi^N_i(\vec{x},E_N)
  =
  Q_i^N(\vec{x},E_N),
\end{equation}
again with no inflow contribution. The associated secondary doses are
computed from the scalar secondary fluences
$\Phi^P=\sum_{i=1}^Q\varpi_i\Psi_i^P$ and $\Phi^N=\sum_{i=1}^Q\varpi_i\Psi_i^N$
as
\begin{equation}
  \label{eq:secondary_doses}
  D_P(\vec{x}) := \int_{I_P} S_P(E_P) \Phi^P(\vec{x},E_P) \d E_P,
  \qquad
  D_N(\vec{x}) := \int_{I_N} \kappa_N(E_N) \Phi^N(\vec{x},E_N) \d E_N,
\end{equation}
where $\kappa_N$ is a prescribed kerma coefficient. The total dose is
$D(\vec{x}) := D_C^{Q}(\vec{x}) + D_P(\vec{x}) + D_N(\vec{x})$.

Figure~\ref{fig:carbon_multispecies_fields} shows the spatial dose
components for a representative configuration. The carbon primary dose
$D_C^{Q}$ exhibits a pronounced Bragg peak localised near the calibrated
range. The forced secondary proton dose $D_P$ is smoother in depth and
typically produces a distal tail beyond the primary peak due to the
volumetric nature of the source and the broad emitted energy spectrum.
The neutron kerma dose $D_N$ forms a low-amplitude halo extending away
from the production region under ballistic transport with removal. The
sum $D$ illustrates how even a reduced multi-species construction can
generate post-peak and out-of-field dose structure beyond a purely
primary CSDA model.

\begin{figure}[h!]
  \centering
  \begin{subfigure}{.48\linewidth}
    \includegraphics[width=\linewidth]{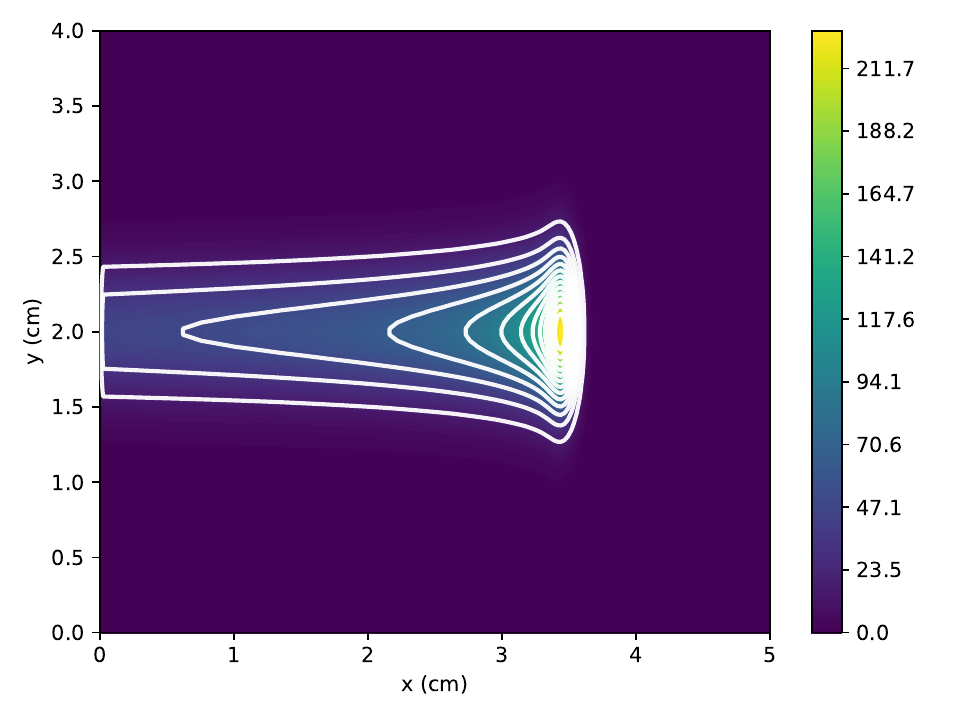}
    \caption{Carbon primary dose $D_C^{Q}$.}
    \label{fig:carbon_primary_dose}
  \end{subfigure}
  \hfil
  \begin{subfigure}{.48\linewidth}
    \includegraphics[width=\linewidth]{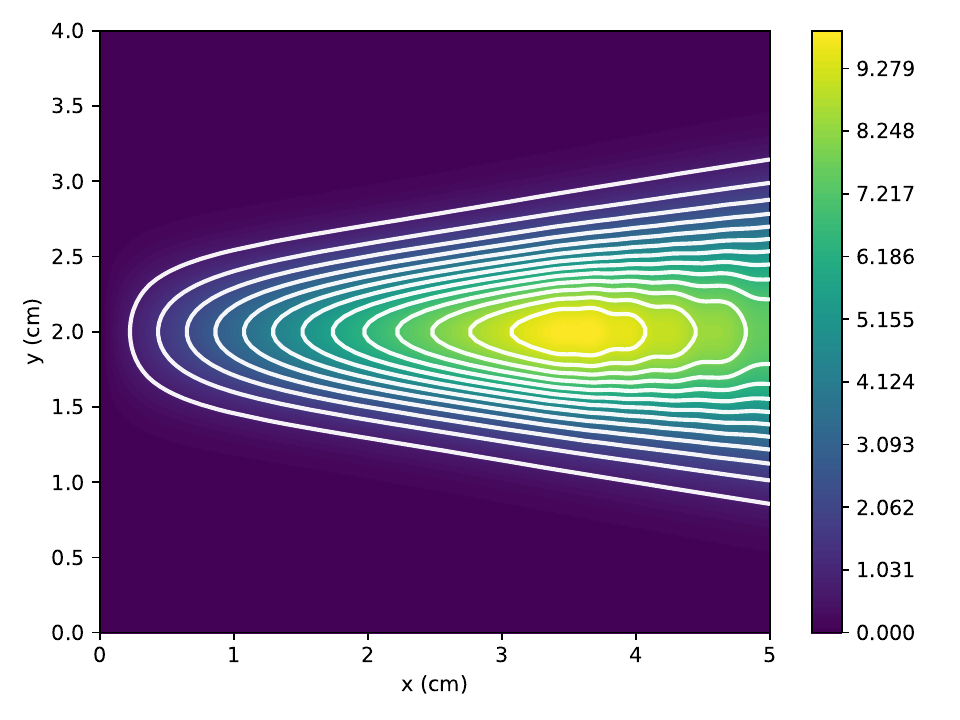}
    \caption{Secondary proton dose $D_P$.}
    \label{fig:carbon_secondary_proton_dose}
  \end{subfigure}

  \vspace{0.6em}
  
  \begin{subfigure}{.48\linewidth}
    \includegraphics[width=\linewidth]{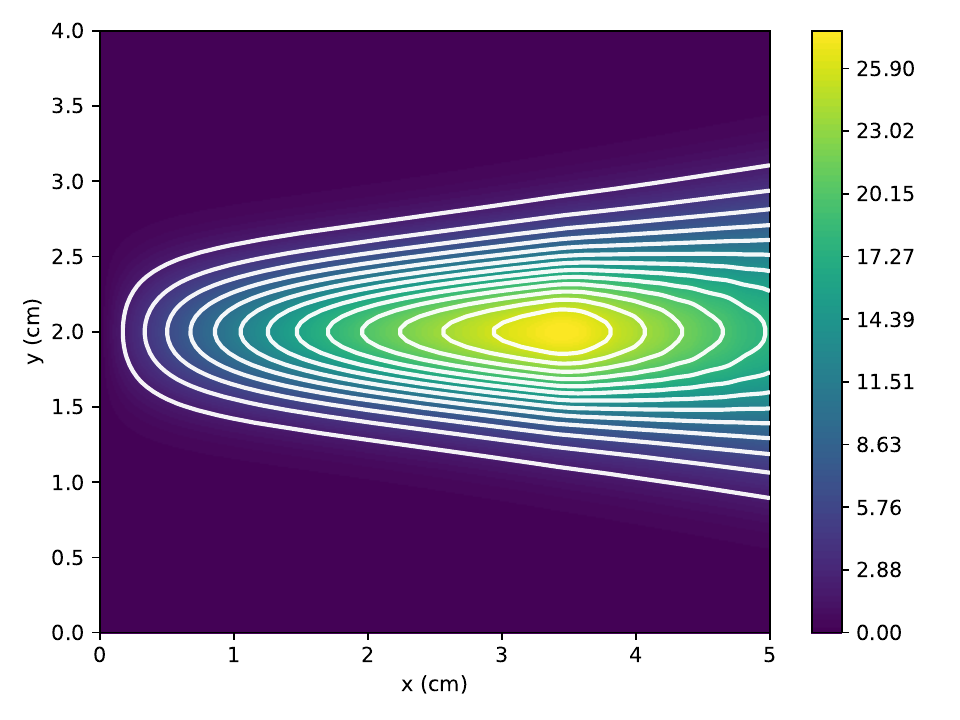}
    \caption{Neutron kerma dose $D_N$.}
    \label{fig:carbon_neutron_kerma_dose}
  \end{subfigure}
  \hfil
  \begin{subfigure}{.48\linewidth}
    \includegraphics[width=\linewidth]{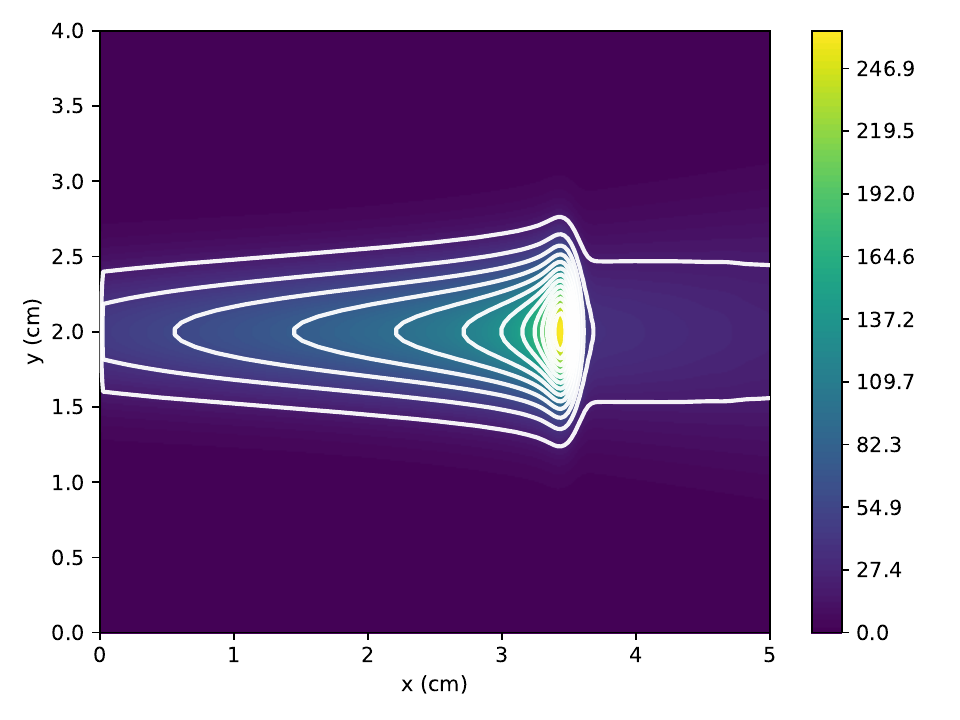}
    \caption{Total dose $D=D_C^{Q}+D_P+D_N$.}
    \label{fig:carbon_total_dose}
  \end{subfigure}
  \caption{\label{fig:carbon_multispecies_fields}
  Multi-species carbon example. Dose components computed from a scattered
  carbon primary field together with forced secondary proton and neutron
  solves driven by the interaction proxy $\mathcal I^C$.
  The secondary components are displayed on the same spatial domain to
  highlight formation of a distal tail and a low-dose halo beyond the
  primary carbon Bragg peak.}
\end{figure}

To make the distal and out-of-field behaviour easier to interpret,
Figure~\ref{fig:carbon_multispecies_depthdose} plots central depth--dose
curves extracted from a narrow band about the mid-plane, using the same
diagnostic averaging as in the proton experiments. The carbon curve
shows primary peak localisation, while the secondary proton contribution
provides a smooth continuation past the peak. The neutron kerma component
remains small in-field but extends further downstream and laterally,
consistent with the streaming--removal model.

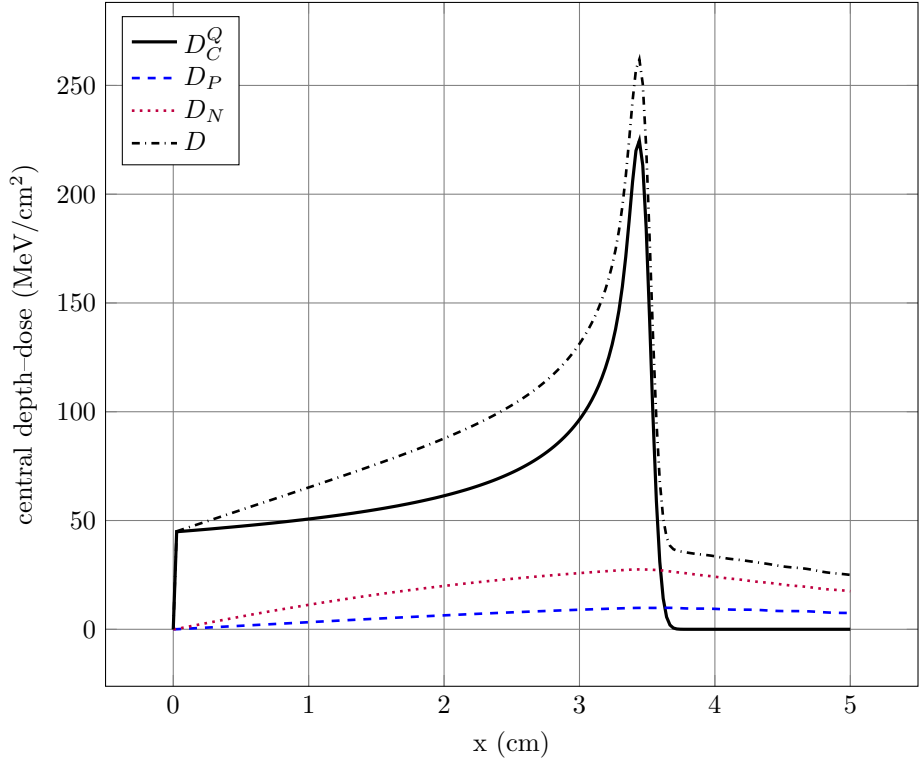
\begin{figure}[h!]
  \centering

  \pgfplotstableread[col sep=comma]{carbon_multispecies_depth_dose.csv}\depthdosedata

  \begin{tikzpicture}
    \begin{axis}[
      width=0.75\linewidth,
      grid=both,
      major grid style={black!50},
      xlabel={x (cm)},
      ylabel={central depth--dose (MeV/cm$^2$)},
      legend style={at={(0.02,0.98)},anchor=north west},
      legend cell align=left,
    ]

      \addplot[
        solid,
        color=black,
        line width=1.2
      ]
      table[x=x_cm, y=D_C]{\depthdosedata};
      \addlegendentry{$D_C^{Q}$}

      \addplot[
        dashed,
        color=blue,
        line width=1.0
      ]
      table[x=x_cm, y=D_P]{\depthdosedata};
      \addlegendentry{$D_P$}

      \addplot[
        dotted,
        color=purple,
        line width=1.0
      ]
      table[x=x_cm, y=D_N]{\depthdosedata};
      \addlegendentry{$D_N$}

      \addplot[
        dashdotted,
        color=black,
        line width=1.0
      ]
      table[x=x_cm, y=D_T]{\depthdosedata};
      \addlegendentry{$D$}

    \end{axis}
  \end{tikzpicture}

  \caption{\label{fig:carbon_multispecies_depthdose} Central
    depth--dose curves for the carbon primary dose $D_C^{Q}$,
    secondary proton dose $D_P$, neutron kerma dose $D_N$ and total
    dose $D=D_C^{Q}+D_P+D_N$, extracted by averaging over a narrow
    band about the mid-plane. The secondary proton component produces
    a smooth distal tail beyond the primary carbon peak, while the
    neutron component extends further under ballistic transport with
    removal and is therefore relevant for out-of-field exposure.}
\end{figure}

\begin{remark}[Scope of secondaries and follow-on modelling]
  Carbon--water nuclear interactions generate many secondary
  species. The present experiment is intentionally reduced. It is
  motivated by out-of-field endpoints, we retain neutrons as the
  principal long-range component and include secondary protons as a
  simple charged-fragment proxy. The source construction
  \eqref{eq:secondary_sources} is designed to be replaced by
  data-driven yields and emitted spectra and extended to additional
  fragment species and protection quantities without changing the
  transport solvers or postprocessing. This is developed further in
  the follow-on study \cite{mmb}.
\end{remark}

\section*{Acknowledgements}

This work was initiated at a workshop hosted by Mathrad, supported by
the EPSRC programme grant EP/W026899/1, which subsequently funded
TP. BA and TP also received support from the Leverhulme Trust grant
RPG-2021-238 and TP the EPSRC grant EP/X030067/1. The research was
conducted by a working group sponsored by the radioprotection theme of
the Institute for Mathematical Innovation partially funding BA.

\printbibliography

\end{document}